\documentclass[a4paper,11pt]{amsart}

\pdfoutput=1

\usepackage[text={400pt,660pt},centering]{geometry}

\usepackage{amsthm, amssymb, amsmath, amsfonts, mathrsfs}
\usepackage{mathtools}
\usepackage{scalerel} 
\usepackage{stackrel}
\usepackage[colorlinks=true, pdfstartview=FitV, linkcolor=blue, citecolor=blue, urlcolor=blue,pagebackref=false]{hyperref}
\usepackage{esint} 
\usepackage{MnSymbol} 
\usepackage{booktabs} 

\usepackage{microtype}
\usepackage{cleveref}

\newtheorem{proposition}{Proposition}
\newtheorem{theorem}[proposition]{Theorem}
\newtheorem{lemma}[proposition]{Lemma}

\theoremstyle{remark}
\newtheorem{remark}[proposition]{Remark}

\theoremstyle{definition}

\numberwithin{equation}{section}
\numberwithin{proposition}{section}
\numberwithin{table}{section}

\renewcommand{\le}{\leqslant}
\renewcommand{\ge}{\geqslant}
\renewcommand{\leq}{\leqslant}
\renewcommand{\geq}{\geqslant}

\renewcommand{\subset}{\subseteq}

\newcommand{\mcl}{\mathcal}
\newcommand{\msf}{\mathsf}
\newcommand{\msc}{\mathscr}

\newcommand{\E}{\mathbb{E}}
\newcommand{\Er}{\mathbb{E}_{\rho}}
\renewcommand{\Pr}{\mathbb{P}_{\rho}}
\renewcommand{\a}{\mathbf{a}}

\newcommand{\ab}{{\overbracket[1pt][-1pt]{\a}}}

\newcommand{\ub}{{\overbracket[1pt][-1pt]{u}}}
\newcommand{\Ub}{{\overbracket[1pt][-1pt]{U}}}

\newcommand{\Ll}{\left}
\newcommand{\Rr}{\right}
\newcommand{\lhs}{left-hand side}
\newcommand{\rhs}{right-hand side}
\newcommand{\1}{\mathbf{1}}
\newcommand{\R}{\mathbb{R}}

\newcommand{\N}{\mathbb{N}}

\newcommand{\Z}{\mathcal{Z}}
\newcommand{\Zd}{{\mathbb{Z}^d}}
\renewcommand{\P}{\mathbb{P}}

\renewcommand{\bar}{\overline}

\renewcommand{\tilde}{\widetilde}

\newcommand{\de}{\delta}
\newcommand{\eps}{\varepsilon}
\newcommand{\ep}{\varepsilon}
\renewcommand{\epsilon}{\varepsilon}
\renewcommand{\d}{{\mathrm{d}}}

\newcommand{\cov}{\mathbb{C}\mathrm{ov}}

\newcommand{\dr}{\partial}

\newcommand{\PP}{\mathbf{P}}
\newcommand{\EE}{\mathbf{E}}

\newcommand{\Pb}{\bar{P}}
\newcommand{\Ee}{\mathcal{E}}
\newcommand{\D}{\mathcal{D}}
\newcommand{\la}{\left\langle}
\newcommand{\ra}{\right\rangle}
\newcommand{\rah}{\right\rangle_{\cH^{-1}, \cH^1_0}}

\newcommand{\cu}{{\scaleobj{1.2}{\square}}}

\renewcommand{\fint}{\strokedint}
\newcommand{\Rd}{{\mathbb{R}^d}}

\newcommand{\mmd}{\mathcal{M}_\delta}
\newcommand{\bmmd}{\bar{\mathcal{M}}_\delta}
\newcommand{\fil}{\mathscr{F}}

\DeclareMathOperator{\dist}{dist}
\DeclareMathOperator{\supp}{supp}
\DeclareMathOperator{\diam}{diam}

\newcommand{\cC}{\msc{C}}   
\newcommand{\cL}{\msc{L}}   
\newcommand{\cH}{\msc{H}}   

\newcommand{\acH}{\underline{\msc{H}}}

\newcommand{\Ind}[1]{\mathbf{1}_{\left\{#1\right\}}}
\newcommand{\id}{\mathsf{Id}}
\newcommand{\norm}[1]{\left\Vert{#1}\right\Vert}
\newcommand{\bracket}[1]{\left\langle{#1}\right\rangle}

\newcommand{\bmu}{\boldsymbol{\mu}}
\newcommand{\EEbmu}{\EE_{\bmu_0}}

\newcommand{\mres}{\mathbin{\vrule height 1.4ex depth 0pt width
		0.13ex\vrule height 0.13ex depth 0pt width 1ex}}

\title[Quantitative fluctuations for interacting particle systems]{Quantitative equilibrium fluctuations for interacting particle systems}
\author[C.\ Gu, 
J.-C.\ Mourrat, M.\ Nitzschner]{Chenlin Gu, Jean-Christophe Mourrat, Maximilian Nitzschner}

\address[Chenlin Gu]{Yau Mathematical Sciences Center, Tsinghua University,   China}

\address[Jean-Christophe Mourrat]{Ecole Normale Sup\'erieure de Lyon and CNRS, Lyon, France}

\address[Maximilian Nitzschner]{Department of Mathematics, The Hong Kong University of Science and Technology, Hong Kong}

\begin{document}

	\begin{abstract}
		
		We consider a class of interacting particle systems in continuous space of non-gradient type, which are reversible with respect to Poisson point processes with constant density. For these models, a rate of convergence was recently obtained in \cite{bulk} for certain finite-volume approximations of the bulk diffusion matrix. Here, we show how to leverage this to obtain quantitative versions of a number of results capturing the large-scale fluctuations of these systems, such as the convergence of two-point correlation functions and the Green--Kubo formula.

		\bigskip
		
		
		\medskip
		
		
	\end{abstract}
	\maketitle
	
	
	%
	%
	%
	%
	%
	%
	%
	%
	\section{Introduction}
	We consider a system of particles $(X_i(t))_{i \in \N, t \ge 0}$ evolving on $\R^d$ with local interactions. We take particles to be indistinguishable, so we encode their positions by keeping track of the measure 
	\begin{equation}
		\label{e.def.mut}
		\bmu_t := \sum_{i \in \N} \delta_{X_i(t)},
	\end{equation}
	an element of the configuration space $\mcl M_\delta(\Rd)$ of locally finite measures that are sums of Dirac masses on $\R^d$. We denote by $\P_\rho$ the law of the Poisson point process with constant intensity $\rho > 0$, which is a probability measure on $\mcl M_\delta(\Rd)$, and we write $\E_\rho$ for the associated expectation. The model we consider is reversible with respect to $\P_\rho$ for every $\rho$, and is of non-gradient type. Precisely, its definition is based on the choice of a measurable function $\a_\circ$ from $\mcl M_\delta(\Rd)$ to the space $\R^{d \times d}_{\mathrm{sym}}$ of $d$-by-$d$ symmetric matrices that satisfies the following properties.
	\begin{itemize}
		\item \emph{Uniform ellipticity}: there exists $\Lambda < +\infty$ such that for every $\mu \in \mcl M_\delta(\Rd)$ and $\xi \in \Rd$,
		\begin{equation}
			\label{e.ass.unif.ell}
			|\xi|^2 \le \xi \cdot \a_\circ(\mu) \xi \le \Lambda |\xi|^2.
		\end{equation}

		\item \emph{Finite range of dependence}: for every $\mu \in \mcl M_\delta(\Rd)$, we have 
		\begin{equation}
			\label{e.finite.range}
			\a_\circ(\mu) = \a_\circ(\mu \mres B_1) .
		\end{equation}
	\end{itemize}
	In \eqref{e.finite.range} and throughout the paper, for every $r > 0$, we write $B_r$ to denote the open Euclidean ball of radius $r \ge 0$ centered at the origin, and for every measure $\mu$ on $\Rd$ and Borel set $A \subset \Rd$, we write $\mu \mres A$ to denote the restriction of the measure $\mu$ to $A$, that is, for every Borel set $B \subset \Rd$,
	\begin{equation*}  
		(\mu \mres A)(B) = \mu(A \cap B). 
	\end{equation*}
	For every $x \in \Rd$, we write $\tau_x$ to denote the translation operator on measures, that is, for every Borel set $B \subset \Rd$,
	\begin{equation*}  
		(\tau_{-x} \mu)(B) := \mu(x + B),
	\end{equation*}
	and we set, for every $\mu \in \mcl M_\de(\Rd)$ and $x \in \Rd$,
	\begin{equation}
		\label{e.def.amu}
		\a(\mu,x) := \a_\circ(\tau_{-x} \mu). 
	\end{equation}
	Roughly speaking, our model is built in such a way that each particle $X_i(t)$, $i \in \N$, follows a diffusive motion whose diffusion matrix at time $t$ is given by 
	\begin{equation*}  
		\a(\bmu_t, X_i(t)) = \a_\circ\big(\tau_{-X_i(t)} \bmu_t\big) = \a_\circ \bigg( \sum_{j \in \N} \delta_{X_j(t) - X_i(t)}\bigg) .
	\end{equation*}
	As will be explained below, one convenient way to give a precise definition of this stochastic process and its associated semigroup $(P_t)_{t \ge 0}$ is to relate these to the Dirichlet form 
	\begin{equation}  
		\label{e.intro.dirichlet.form}
		f \mapsto \E_\rho \Ll[ \int_\Rd \frac 1 2 \nabla f \cdot \a \nabla f \, \d \mu \Rr] ;
	\end{equation}
	we refer to~\eqref{e.def.deriv} for a definition of the gradient $\nabla$, and to Section~\ref{subsec.Dirichletform} for a detailed construction. We denote by $\PP_{\bmu_0}$ the law of the cloud of particles $(\bmu_t)_{t \ge 0}$ started from $\bmu_0$, and by $\EE_{\bmu_0}$ the associated expectation. A classical object of study is the rescaled density fluctuation field
	\begin{equation}  
		\label{e.def.YtL}
		Y_t^N := N^{-\frac d 2}(\bmu_{N^2 t} - \rho m)(N\, \cdot),
	\end{equation}
	where $N \ge 1$ and $m$ denotes the Lebesgue measure on $\Rd$. We postulate that the initial configuration $\bmu_0$ is sampled according to $\P_\rho$; since we denote by~$\mu$ the canonical random variable under $\E_\rho$, we implicitly set $\mu = \bmu_0$ in expressions such that $\Er \EEbmu$. One can think of $Y_t^N$ as a (random) distribution over $\Rd$, and testing $Y_t^N$ against a smooth function $f \in C^\infty_c(\Rd;\R)$ of compact support results in
	\begin{equation}  
		\label{e.explicit.def.YtL}
		Y_t^N(f) = N^{-\frac d 2} \Ll( \int_\Rd f(N^{-1} x) \, \d \bmu_{N^2 t}(x) - \rho \int_\Rd f(N^{-1} x) \, \d x  \Rr) . 
	\end{equation}
	Using the invariance of the measure $\P_\rho$ under the dynamics, one can easily check that for each fixed $t \ge 0$, the distribution $Y_t^N$ converges in law under $\Pr \PP_{\bmu_0}$ to a white noise over~$\Rd$ with variance $\rho$ as $N$ tends to infinity. The space-time correlations are less direct and are governed by the bulk diffusion matrix $\ab \in \R^{d \times d}_{\mathrm{sym}}$, whose precise definition is given in~\eqref{eq:AbarDef}. For a discrete class of models analogous to our present setting, it is shown in \cite{cha96, fun96, lu94} that, as a space-time distribution, $(Y^N_t)_{t \ge 0}$ converges in law to~$(Y^\infty_t)_{t \ge 0}$ solution to
	\begin{equation}
		\label{e.eq.Yinfty}
		\dr_t Y^\infty_t = \frac 1 2 \nabla \cdot \ab \nabla Y^\infty_t + \nabla \cdot \Ll( \sqrt{\rho \ab} \xi_t \Rr) ,
	\end{equation}
	where $(\xi_t(x))_{t \ge 0, x \in \Rd}$ is a $d$-dimensional space-time white noise over $[0,+\infty) \times \Rd$. In particular, we expect the convergence of two-point density correlation functions, in the sense that for every $f, g \in C^\infty_c(\Rd;\R)$ and $t > s > 0$,
	\begin{equation}
		\label{e.intro.conv.two.point}
		\lim_{N \to + \infty} \E_\rho\EE_{\bmu_0} [Y^N_t(f) Y^N_s(g)]= \E[Y^\infty_t(f) Y^\infty_s(g)].
	\end{equation}
	The latter quantity can be rewritten explicitly as
	\begin{equation}
		\label{e.def.two.point.Yinfty}
		\E[Y^\infty_t(f) Y^\infty_s(g)] = \rho \int_{\Rd \times \Rd} f(x)\Psi_{t-s}(x-y)g(y)\, \d x  \d y,    
	\end{equation}
	where $(\Psi_t(x))_{t \ge 0, x \in \Rd}$ is the heat kernel for the operator $\partial_t - \frac 1 2 \nabla \cdot \ab \nabla$, that is,
	\begin{equation}
		\label{e.def.heat.kernel}
		\Psi_t(x) := \frac{1}{\sqrt{(2\pi t)^{{d}} \det \ab}}\exp\Ll(-\frac{x \cdot \ab^{-1} x}{2 t}\Rr).
	\end{equation}
	Our first main result gives a rate of convergence for \eqref{e.intro.conv.two.point}. We prefer to state it using the unrescaled centered field defined by
	\begin{equation}
		\label{e.def.Y}
		Y_t := Y^{1}_t = \bmu_t - \rho m,
	\end{equation}
	and for every $f \in L^1(\Rd) \cap L^2(\Rd)$, we write
	\begin{equation}
		\label{e.def.Y.f}
		Y_t(f) = \int_\Rd f \, \d \bmu_t - \rho \int_\Rd f \, \d m.
	\end{equation}
	\begin{theorem}[Quantitative asymptotics for two-point functions]
		\label{t.two.point}
		There exists an exponent $\beta(d,\Lambda,\rho) > 0$ and a constant $C(d,\Lambda,\rho) < + \infty$ such that for every $f, g \in L^1(\Rd) \cap L^2(\Rd)$ and $t > s > 0$, we have
		\begin{multline}
			\label{e.two.point}
			\Ll|\E_\rho \EEbmu \Ll[ Y_t(f) Y_s(g) \Rr] -\rho \int_{\Rd \times \Rd} f(x)\Psi_{t-s}(y-x)g(y)\, \d x  \d y \Rr|
			\\
			\le C (t-s)^{-\beta} \|f\|_{L^2} \, \|g\|_{L^2}. 
		\end{multline}	
	\end{theorem}
	By a simple scaling, Theorem~\ref{t.two.point} yields that for every $f, g \in L^1(\Rd) \cap L^2(\Rd)$, we have 
	\begin{equation}\label{e.two.pointScale}  
		\Ll|  \E_\rho\EEbmu [Y^N_t(f) Y^N_s(g)]- \E [Y^\infty_t(f) Y^\infty_s(g)] \Rr| \le C (N^2(t-s))^{-\beta} \|f\|_{L^2} \, \|g\|_{L^2}. 
	\end{equation}

	As another illustration of our results, we show a quantitative version of the Green--Kubo formula. Standard versions of the Green--Kubo formula involve infinite-time and infinite-volume limits. Here we localize the time and space scales, with an estimate of the error. For convenience, we localize in space using boxes of the form
	\begin{equation}\label{eq.defcube}  
		\cu_m := \Ll( -\frac{3^{m}}{2}, \frac{3^{m}}{2} \Rr) ^d, \quad m \in \N.
	\end{equation}
	The Green--Kubo formula involves an integral of the correlations of the current of the particle density, as for instance in \cite[Proposition~II.2.1]{spohn2012large}. Informally, for every $\mu \in \mmd(\Rd)$, the $\EE_\mu$-averaged instantaneous current at a point $x \in \supp \mu$ in the direction of $p \in \R^d$ should be given by
	\begin{equation*}  
		\frac{1}{2}\nabla_x \cdot(\a(\mu,x)p) = \frac{1}{2}\nabla_x \cdot (\a(\mu' + \de_x, x) p), \qquad \text{ with } \mu' := \mu - \de_x.
	\end{equation*}
	This is a continuous analogue of the quantity 
	$j_l(\eta)$ in \cite[(II.2.23)]{spohn2012large}, with the correspondence of $(l,\eta,0)$ in~\cite{spohn2012large} to $(p,\mu,x)$ in our notation. Since we would like to state a version of the Green--Kubo formula in finite volume, it would be tempting to define the total current over the box $\cu_m$ as
	\begin{equation*}  
		\int_{\cu_m} \frac{1}{2}\nabla_x \cdot (\a(\mu,x) p) \, \d \mu(x).
	\end{equation*}
	However, even if we leave aside the possible difficulties associated with a lack of smoothness of the mapping $x \mapsto \a(\mu' + \de_x, x)$, this definition fails to capture an important boundary contribution. Instead, we define the total current over the box $\cu_m$ by duality as follows. We let $\cL^2(\cu_m)$ be the space of square-integrable functions which are measurable with respect to the configuration in $\cu_m$; see also \eqref{eq.defL2}. We denote by $\cH^1_0(\cu_m)$ the space of functions with square-integrable gradients that can be approximated by smooth functions that only depend on $\mu \mres K$ for some compact set $K \subset \cu_m$; see below~\eqref{eq.defNormalH1} for a precise definition. We denote by $\cH^{-1}(\cu_m)$ the space dual to $\cH^{1}_0(\cu_m)$, with the canonical embedding
	\begin{equation*}  
		\cH^{1}_0(\cu_m) \subset \cL^2(\cu_m) \subset \cH^{-1}(\cu_m),
	\end{equation*}
	and by $\la \cdot , \cdot \ra_{\cH^{-1}(\cu_m), \cH^1_0(\cu_m)}$ the duality pairing between $\cH^{-1}(\cu_m)$ and $\cH^1_0(\cu_m)$. 
	For every $p \in \Rd$ and $m \in \N$, we define $F_{p,m} \in \cH^{-1}(\cu_m)$, the integrated current over~$\cu_m$, by duality, so that for every $g \in \cH^1_0(\cu_m)$, we have
	\begin{equation}\label{eq.defFpm}  
		\la F_{p,m} , g \ra_{\cH^{-1}(\cu_m), \cH^1_0(\cu_m)} := \Er \Ll[\int_{\cu_m} - \frac{1}{2} p \cdot \a \nabla g \, \d \mu\Rr].
	\end{equation}
	As is the case for diffusions on bounded subsets of $\Rd$, we need to choose a boundary condition in order to define the particle dynamics in finite volume, for instance of Dirichlet or Neumann type. In the context of particle systems, the ``Dirichlet''-type boundary condition is understood in the sense provided by the space $\cH^1_0(\cu_m)$. Heuristically, with this boundary condition, particles that hit the boundary instantaneously disappear, and some particles also appear spontaneously at the boundary to maintain detailed balance. Since this in effect places the finite-volume dynamics ``within a bath of particles at equilibrium'', we feel that this is a more appropriate dynamics for the estimation of the total flux than a Neumann-type boundary condition (which would correspond to dynamics with a conserved number of particles that reflect at the boundary). We denote by $(\bmu_t^{(m)})_{t \ge 0}$ the particle dynamics restricted to the box $\cu_m$ with ``Dirichlet'' boundary condition, and by $(P_t^{(m)})_{t \ge 0}$ the associated semigroup, so that $P_t^{(m)}(f)(\bmu_0) = \EEbmu [f(\bmu_t^{(m)})]$; these objects are defined more precisely in
	Section~\ref{sec:5}. In analogy with \cite[(II.2.27)]{spohn2012large} (which we think should display a term $x_\alpha x_\beta$ in place of $\delta_{\alpha \beta} x_\alpha^2$), \cite[(1)]{sasada2018green}, or \cite[(2.2), (2.4), (8.7), (8.9)]{mourrat2019efficient}, our qualitative Green--Kubo formula takes the form 
	
	\begin{multline}
		\label{e.gk.quali}
		\frac{1}{2} p \cdot \ab p = \frac{1}{2}\Er[p\cdot \a_\circ(\mu + \delta_0) p] \\
		- \lim_{\lambda \to 0} \lim_{m \to \infty} \frac{1}{\rho |\cu_m|} \int_0^{+\infty}  e^{-\lambda t}  \la F_{p,m}, P_t^{(m)}(F_{p,m}) \rah \, \d t.
	\end{multline}
	Here we write $\la \cdot , \cdot \ra_{\cH^{-1}, \cH^1_0}$ as a shorthand notation for $\la \cdot , \cdot \ra_{\cH^{-1}(\cu_m), \cH^1_0(\cu_m)}$. Our second main result provides quantitative estimates that allow us to localize this formula in space and time. In particular, we can take $\lambda = 0$ there, in which case the integral of the current correlation becomes the classical case without any localization in time.
	\begin{theorem}[Quantitative Green--Kubo formula]
		\label{t.quanti.gk}
		There exist $\alpha(d,\Lambda,\rho) > 0$ and $C(d, \Lambda, \rho) < + \infty$ such that for every $m \in \N$, $p \in B_1$ and $\lambda \in [0, +\infty)$, we have 
		\begin{multline}\label{eq.quanti.gk}
			\Ll|\frac{1}{2} p \cdot \ab p - \frac{1}{2}\Er[p\cdot \a_\circ(\mu + \delta_0) p] + \frac{1}{\rho |\cu_m|} \int_0^{+\infty}  e^{-\lambda t}  \la F_{p,m}, P_t^{(m)}(F_{p,m}) \rah \, \d t \Rr|
			\\
			\leq C \Ll\{
			\begin{array}{ll}
				1 \qquad & \text{if } \lambda \in [1, +\infty),\\
				\lambda^{\frac{\alpha}{2(1+\alpha)}} \qquad & \text{if } \lambda \in (3^{-2(1+\alpha)m}, 1),\\
				3^{-\alpha m}&  \text{if } \lambda \in[0, 3^{-2(1+\alpha)m}].
			\end{array} \Rr. 
		\end{multline}
	\end{theorem}

	\subsection*{Related works} We now give a brief review of related works. We already mentioned \cite{cha96, fun96, lu94} which identified the scaling limit of the rescaled fluctuation field $Y^N$ for discrete models of non-gradient type similar to our continuous model. Similar results were also obtained in \cite{broxrost,cha94,chayau92, mpsw86, spo86} for other models. The hydrodynamic limit of some non-gradient models was obtained in \cite{klo94, quastel, varadhanII} using the entropy method introduced in \cite{gpv}, and in \cite{fuy} using the relative entropy method introduced in~\cite{yau}. Reference books on the topic include \cite{kipnis1998scaling, komorowski2012fluctuations, spohn2012large}. 
	
	We are not aware of any result addressing the out-of-equilibrium fluctuations of non-gradient models. While we focus here on the derivation of quantitative estimates for the fluctuations at equilibrium, we hope that our techniques will also be useful for at least some out-of-equilibrium situations. In relation to this, the fact  that the bulk diffusion matrix $\ab$ depends smoothly on the density $\rho$ has been shown in \cite{giunti2021smoothness}. For gradient models and small perturbations thereof, out-of-equilibrium fluctuation results have been derived in \cite{chayau92, mfl86, fpv88, jarmen18, prespo83, ravishankar1992fluctuations}.
	Recent progress on higher-order approximations and large deviations include \cite{cornalba2023dean, dirr2020conservative, fehrman2023nonequilibrium}. 
	
	The results of the present paper are based on the quantitative estimates obtained in \cite{bulk}. The approach taken up there is inspired by prior work on the homogenization of elliptic equations, as reviewed in \cite{armstrong2022elliptic, AKMbook, informal}. 
	
	Closely related to the problem investigated here is the question of obtaining quantitative estimates on the relaxation to equilibrium. One way to measure this is to estimate the rate of convergence to zero of quantities of the form
	\begin{equation*}  
		\Er \Ll[ (P_t f)^2 \Rr] = \Er \Ll[ f \, (P_{2t} f) \Rr],
	\end{equation*}
	for suitable centered functions $f : \mmd(\Rd)\to \R$, and where we used the reversibility and the semigroup property of $(P_t)_{t \ge 0}$ to derive the identity above. For specific choices of the observable $f$, the statement of Theorem~\ref{t.two.point} gives an upper bound on the next-order correction to the leading-order behavior of $\Er \Ll[ (P_t f)^2 \Rr]$. For different classes of functions $f$, the leading-order behavior was investigated in \cite{gu2020decay} for the model we also consider here. Other models were considered in \cite{berzeg, ccr, deu94, jlqy, lanyau, lig91}; see also \cite{vardecay} that relates this sort of problem with the quantitative homogenization of elliptic equations. Heat-kernel estimates for a tagged particle evolving in a simple exclusion process were obtained in \cite{giunti2019heat}.

	\subsection*{Organization of the paper} The remainder of this article is organized as follows. In Section~\ref{sec:2}, we introduce further notation and recall several results from~\cite{bulk}, as well as the semigroups on the configuration space. In Section~\ref{sec:3}, we develop a two-scale expansion for elliptic equations on the configuration space, which we find interesting in its own right and illustrates our proof approach in the simplest possible setting. Section~\ref{sec:4} establishes Theorem~\ref{t.two.point} using a two-scale expansion for parabolic equations. Finally, in Section~\ref{sec:5} we prove the quantitative error estimates in the Green--Kubo formula stated in Theorem~\ref{t.quanti.gk}.
	
	\section{Preliminaries}
	\label{sec:2}
	\subsection{Euclidean space}
	
	We denote by $\R^d$, $d\geq 1$, the standard Euclidean space, and as stated in \eqref{eq.defcube}, we denote by~$\cu_m$ the open hypercube of side-length~$3^m$ centered at the origin. For every $n \leq m \in \N$, we set
	\begin{equation}\label{e.Z.def}
		\Z_{m,n}  := 3^n \Zd \cap \cu_m, \qquad \Z_n   := 3^n  \Zd.
	\end{equation}
	Then we can partition a large cube $\cu_m$ into cubes of a smaller scale as $\bigcup_{z \in \Z_{m,n}} (z + \cu_n)$, up to a set of null Lebesgue measure, and similarly, partition $\Rd$ into $\bigcup_{z \in \Z_{n}} (z + \cu_n)$, up to a set of null Lebesgue measure.
	
	For any open set $U \subset \Rd$, we write $L^p(U)$ with $p \geq 1$ for the classical space of functions $f$ for which $|f|^p$ has a finite Lebesgue integral over $U$, and $H^k(U)$ with $k \geq 1$ for the classical Sobolev space of order $k$ on $U$. We let $H^1_0(U)$ stand for the closure in $H^1(U)$ of the set of smooth functions with compact support in $U$. For any non-empty and bounded open set $U$, and $f \in L^1(U)$, we also introduce shorthand notation for the Lebesgue integral of $f$, normalized by the Lebesgue measure of $U$, as
	\begin{equation*}  
		\fint_U  f := \frac{1}{|U|} \int_U f.
	\end{equation*}

	\subsection{Configuration space}
	
	The configuration space $\mmd(\Rd)$ and the Poisson point process $\Pr$ are defined at the beginning of the introduction. For a Borel set $U\subset \R^d$,  we write $\mcl F_U$ for the $\sigma$-algebra generated by the mappings $\mu \mapsto \mu(V)$, for all Borel sets $V \subset U$, completed with all the $\P_{\rho}$-null sets. We also use the shorthand notation $\mcl F$ for $\mcl F_{\Rd}$. The $\sigma$-algebra $\mcl F$ is also the Borel $\sigma$-algebra associated with the topology of vague convergence on $\mmd(\Rd)$, as explained for instance in \cite[Exercise~5.2 and solution]{HJbook}.

	We introduce several function spaces on $\mmd(\Rd)$ that will be used below. On the configuration space $\mmd(\Rd)$, we define 
	\begin{align}\label{eq.defL2}
		\cL^2(U) := L^2(\mmd(\Rd), \mcl F_U, \Pr),
	\end{align}
	the complete space of the $\mcl F_U$-measurable functions with finite second moment under $\Pr$. We write $\cL^2 := \cL^2(\Rd)$ as a shorthand notation. It is a Hilbert space and we denote by $\bracket{\cdot , \cdot}_{\cL^2}$ its associated inner product.

	For every sufficiently smooth function $f : \mmd(\Rd) \to \R$, measure $\mu \in \mmd(\Rd)$, and $x \in \supp \mu$, the gradient $\nabla f(\mu,x)$ is defined by requiring, for every ${k \in \{1,\ldots,d\}}$, that
	\begin{equation}  
		\label{e.def.deriv}
		\mathrm e_k \cdot \nabla f(\mu, x) = \lim_{h \to 0} \frac{f(\mu - \delta_x + \delta_{x + h \mathrm e_k}) - f(\mu)}{h},
	\end{equation}
	where $(\mathrm e_1,\ldots, \mathrm e_d)$ denotes the canonical basis of $\Rd$. 
		We define the spaces of smooth functions $\cC^\infty(U)$ and $\cC^{\infty}_c(U)$ for an open set $U \subset \Rd$ as follows. For every open set $V \subset \Rd$ and any $\mcl F$-measurable function $f$, suppose the configuration outside of $V$ is fixed and the number of particles in $V$ is given by $\mu(V) = n \in \N$. This naturally gives rise to a measurable canonical projection $(x_1,...,x_n) \mapsto f_n(x_1,...,x_n, \mu \mres V^c) := f(\sum_{i = 1}^n \delta_{x_i} + \mu \mres V^c)$, defined on $V^n$. We define $\cC^\infty(U)$ as the space of $\mcl F$-measurable functions $f : \mmd(\Rd) \to \R$ such that for every bounded open set $V \subset U$, $\mu \in \mmd(\Rd)$, and $n \in \N$, the canonical projection $f_n(\cdot, \mu \mres V^c)$ is infinitely differentiable on $V^n$. The space $\cC^{\infty}_c(U)$ is the subspace of $\cC^{\infty}(U)$ of functions $f$ for which there exists a compact set $K \subseteq U$ such that $f$ is $\mcl F_K$-measurable.

		We now define $\cH^1(U, \Pr)$, an infinite dimensional analogue of the classical  Sobolev space $H^1$. For every $f \in \cC^\infty(U)$, we introduce the norm
		\begin{align}\label{eq.defH1}
			\norm{f}_{\cH^1(U, \Pr)} := \Ll(\Er[f^2(\mu)] + \Er\Ll[\int_U \vert \nabla f(\mu, x) \vert^2 \, \d \mu(x) \Rr]\Rr)^{\frac{1}{2}},
		\end{align}
		and we define the space $\cH^1(U, \Pr)$ as the completion of the set of functions $f \in \cC^\infty(U)$ such that $\norm{f}_{\cH^1(U, \Pr)}$ is finite, with respect to this norm. Since the density $\rho$ in this paper is kept fixed, we write $\cH^1(U)$ as a shorthand notation for $\cH^1(U, \Pr)$. We also denote by $\acH^1(U)$ a normalized version of this norm which is convenient in order to absorb the diameter, entering as a factor in Poincar\'e's inequality (see~\eqref{e.Poincare} below), 
		\begin{align}\label{eq.defNormalH1}
			\norm{f}_{\acH^1(U)} := \Ll(\vert U\vert^{-\frac{2}{d}} \Er[f^2(\mu)] + \Er\Ll[\int_U \vert \nabla f(\mu, x) \vert^2 \, \d \mu(x) \Rr]\Rr)^{\frac{1}{2}}.
		\end{align}
		We stress that {functions in $\cH^1(U)$ need not be $\mcl F_U$-measurable}. Indeed, the function $f$ can depend on $\mu \mres U^c$ in a relatively arbitrary (measurable) way, as long as $f \in \cL^2$. If $V \subset U$ is another open set, then $\cH^1(U) \subset \cH^1(V)$. 

		We define the space $\cH^1_0(U)$ as the closure in $\cH^1(U)$ of the space of functions $f \in \cC^\infty_c(U)$   such that $\norm{f}_{\cH^1(U)}$ is finite. Notice in particular that, in contrast with functions in $\cH^1(U)$, a function in $\cH^1_0(U)$ does not depend on $\mu \mres U^c$. If $V \subset U$ is another open set, then $\cH^1_0(V) \subset \cH^1_0(U)$.

		We also recall a version of Poincar\'e's inequality on the configuration space, which was proved in~\cite[Proposition 3.3]{bulk}: there exists a constant $C(d) < + \infty$ such that for every bounded open set $U \subseteq \R^d$ and $f \in \cH_0^1(U)$, one has
		\begin{equation}
			\label{e.Poincare}
			\Er\left[(f - \Er[f])^2 \right] \leq C \diam(U)^2 \Er\left[\int_U |\nabla f|^2 \d \mu \right],
		\end{equation}
		where $\diam(U)$ denotes the Euclidean diameter of $U$.
		
		We denote by $\cH^{-1}(U)$ the dual space of $\cH^1_0(U)$, which is the space of all continuous linear forms on $\cH^1_0(U)$. We denote by $\bracket{\cdot,\cdot}_{\cH^{-1}(U), \cH^1_0(U)}$ the duality pairing, and define the norm $\norm{\cdot}_{\cH^{-1}(U)}$ as
		\begin{align}\label{eq.defHdual}
			\norm{g}_{\cH^{-1}(U)} := \sup_{f \in \cH^1_0(U), \norm{f}_{\cH^1(U)} \le 1} \bracket{g,f}_{\cH^{-1}(U), \cH^1_0(U)}.
		\end{align}

		\subsection{Homogenization of the bulk diffusion matrix}
		
		In this subsection, we recall some results from \cite{bulk}, including the definition of the bulk diffusion matrix $\ab$ and its quantitative convergence from finite-volume approximations $\ab(\cu_m)$ and $\ab_\ast(\cu_m)$.
		
		For a bounded domain $U \subset \Rd$, vectors $p,q \in \Rd$, and $\rho > 0$, we consider the following optimization problems
		\begin{equation}\label{eq:defNu}
			\begin{split}
				\nu(U,p,\rho) &:= \inf_{v \in \ell_{p,U} + \cH^1_0(U)}\Er \Ll[ \frac{1}{\rho \vert U \vert} \int_{U} \frac{1}{2} \nabla v \cdot \a \nabla v \, \d \mu \Rr], \\
				\nu^*(U,q,\rho) &:=  \sup_{u \in \cH^1(U)} \Er \Ll[\frac{1}{\rho \vert U \vert}\int_{U} \Ll( -\frac 1 2 \nabla u \cdot \a \nabla u + q \cdot \nabla u \Rr) \, \d \mu \Rr],
			\end{split}
		\end{equation}
		with $\ell_{p,U} = \int_U p \, \cdot \, x \, \d \mu(x)$. By~\cite[Proposition 4.1]{bulk}, there exist symmetric $(d\times d)$-matrices $\ab(U,\rho), \ab_\ast(U,\rho)$ satisfying ${\id \leq \ab(U,\rho) \leq \Lambda \id}$ and ${\id \leq \ab_\ast(U,\rho) \leq \Lambda \id}$ such that for every $p,q \in \Rd$ and $\rho > 0$,
		\begin{equation}\label{eq:MatrixEq}
			\nu(U,p,\rho) = \frac{1}{2} p \cdot \ab(U,\rho) p, \quad  \quad \nu^\ast(U,q,\rho) = \frac{1}{2} q \cdot \ab_\ast^{-1}(U,\rho) q.
		\end{equation}
		The sequence $(\ab(\cu_m,\rho))_{m \in \N}$ is decreasing in $m$ (in the sense that $\ab(\cu_{m},\rho) - \ab(\cu_{m+1},\rho)$ is positive semi-definite for every $m\in \N$), and we denote its limit by
		\begin{equation}
			\label{eq:AbarDef}
			\ab(\rho) := \lim_{m \rightarrow \infty} \ab(\cu_m,\rho).
		\end{equation}
		We also refer to \cite[Appendix~B]{bulk} for equivalent definitions of $\ab(\rho)$. 
		The main result of~\cite{bulk} is a quantification of the speed of convergence in~\eqref{eq:AbarDef}. More precisely, by the proof of~\cite[Theorem 2.1]{bulk}, there exist constants $\alpha = \alpha(d,\Lambda,\rho) > 0$ and $C = C(d,\Lambda,\rho) < + \infty$ such that
		\begin{equation}
			\label{eq:QuantConvergence}
			|\ab(\cu_m,\rho) - \ab(\rho)|+ |\ab_\ast(\cu_m,\rho) - \ab(\rho)| \leq C 3^{-\alpha m}. 
		\end{equation}
		Concerning the optimization problems in \eqref{eq:defNu}, we write
		\begin{equation}
			\begin{split}
				v(\cdot,U,p) \in \ell_{p,U} + \cH^1_0(U) &\text{ for the minimizer of $\nu(U,p,\rho)$, and} \\
				u(\cdot,U,q) \in \cH^1(U) &\text{ for the maximizer of $\nu^\ast(U, q,\rho)$.}
			\end{split}
		\end{equation}
		As observed in \cite[Proposition~4.1]{bulk}, these optimizers are unique provided that we also impose that $\Er[v(\mu,U,p)] = 0$ and that $\Er[u(\mu,U,q)| \mu(U), \mu \mres U^c] = 0$. 	
		By \cite[(72)]{bulk}, the first-order optimality conditions read as
		\begin{equation}
			\label{eq:VariationalFormula72}
			\begin{split}
				\forall v' \in \cH^1_0(U), \qquad \E_\rho\left[\int_U \nabla v(\cdot,U,p) \cdot \a \nabla v'  \d \mu\right] &= 0,\\
				\forall u' \in \cH^1(U), \qquad \E_\rho\left[\int_U \nabla u(\cdot,U,q) \cdot \a \nabla u'  \d \mu\right] &= \E_\rho\left[\int_U  q\cdot \nabla u' \d \mu \right].
			\end{split}
		\end{equation}
		By removing the affine part from $v(\cdot,U,p)$ and $u(\cdot,U,q)$, we define the \textit{approximate correctors}
		\begin{align}\label{eq.corrector}
			\phi_{p, U} := v(\cdot,U,p) - \ell_{p,U}, \qquad \phi^*_{p, U} := u(\cdot,U,\ab_\ast(U,\rho)p) - \ell_{p,U}.
		\end{align}
		Here the definition of $\phi^*_{p, U}$ is motivated as follows. As seen in \cite[(65)]{bulk}, we can combine \eqref{eq:VariationalFormula72} and \eqref{eq:MatrixEq} to see that the averaged slope of $ u(\cdot,U,q)$ is $\ab_\ast^{-1}(U,\rho) q$:
		\begin{align}\label{eq.slopeDuel}
			\forall q' \in \Rd, \qquad \E_\rho\left[ \frac{1}{\rho \vert U \vert}\int_U  q'\cdot \nabla u(\cdot,U,q) \d \mu \right] = q' \cdot \ab_\ast^{-1}(U,\rho) q.
		\end{align}
		Therefore, the function $u(\cdot,U,\ab_\ast(U,\rho)p)$ has an average slope of $p$. 
		
		The following lemma subsumes important properties of correctors that will be instrumental in the following sections. Its proof essentially follows from~\cite[Sections~4 and~5]{bulk}.
		\begin{lemma}\label{lem.JBound}
			(1) For every $n \in \N$, $z,z' \in \mathcal{Z}_{n}$ with $\dist(z,z') > 3^n$, and $i \in \{1,\cdots,d\}$, the quantities $\phi_{e_i,z + \cu_n}$ and  $\phi_{e_i,z' + \cu_n}$ are independent random variables.
			
			(2) There exist constants $\alpha(d,\Lambda,\rho) > 0$ and $C(d,\Lambda,\rho) < + \infty$  such that for every $n\in \N$ and $i \in \{1,\cdots,d\}$, we have
			\begin{equation}\label{eq.JBound}
				\Er\Ll[\frac{1}{\rho|\cu_n|}\int_{\cu_n} \vert \nabla \phi_{e_i, \cu_n} - \nabla \phi^*_{e_i, \cu_n} \vert^2 \, \d \mu\Rr] \leq C3^{-  2\alpha n}.
			\end{equation}
		\end{lemma}
		\begin{proof}
			The claim (1) follows directly from observing that $\phi_{e_i,z + \cu_n}$ and  $\phi_{e_i,z' + \cu_n}$ are in $\cH^1_0(z + \cu_n)$ and $\cH^1_0(z' + \cu_n)$ respectively.
			
			Concerning the claim (2), for every $n \in \N$ and $p \in \R^d$, we introduce the quantity
			\begin{equation}\label{e.def.J}
				\begin{split}
					J(\cu_n,p,\ab_\ast(\cu_n,\rho) p) &:= \nu(\cu_n,p,\rho) + \nu^*(\cu_n,\ab_\ast(\cu_n,\rho)p,\rho ) - p \cdot \ab_\ast(\cu_n,\rho)p \\
					& = \frac{1}{2} p \cdot \ab(\cu_n,\rho)p - \frac{1}{2} p \cdot \ab_\ast(\cu_n,\rho) p.
				\end{split}
			\end{equation}
			By~\cite[Proposition 4.2, (86) and (87)]{bulk}, we see that upon writing
			\begin{equation}\label{eq.wDifference}
				\begin{split}
					w(\mu,\cu_n,e_i) &:= u(\mu,\cu_n,\ab_\ast(\cu_n,\rho)e_i) - v(\mu,\cu_n,e_i)\\
					&= \phi_{e_i, \cu_n} - \phi^*_{e_i, \cu_n},
				\end{split}
			\end{equation}
			the quantity $J$ admits the quadratic representation
			\begin{equation}\label{eq.JQuadra}
				J(\cu_n,e_i,\ab_\ast(\cu_n,\rho) e_i) = \Er\Ll[\frac{1}{\rho \vert \cu_n \vert} \int_{\cu_n} \frac{1}{2} \nabla  w(\mu,\cu_n,e_i)  \cdot \a  \nabla  w(\mu,\cu_n,e_i)  \, \d \mu \Rr].
			\end{equation}
			Here we remark that our $w$ is different from that in \cite[(86)]{bulk} up to a term that may only depend on $\mu(\cu_n)$ and $\mu \mres \cu_n^c$, but this does not change the value of the quadratic form on the right side of \eqref{eq.JQuadra}. Combining \eqref{eq.wDifference}, \eqref{eq.JQuadra}, and \eqref{e.ass.unif.ell},  we obtain that
			\begin{equation}\label{e.UpperBound.Corrector.Diff}
				\Er\Ll[\frac{1}{\rho|\cu_n|}\int_{\cu_n} \vert \nabla \phi_{e_i, \cu_n} - \nabla \phi^*_{e_i, \cu_n} \vert^2 \, \d \mu\Rr]  \leq C J(\cu_n,e_i,\ab_\ast(\cu_n,\rho) e_i).
			\end{equation}
			On the other hand, by~\cite[Section 5.4]{bulk}, one knows that
			\begin{equation}
				\label{e.UpperBound.J}
				J(\cu_n,e_i,\ab_\ast(\cu_n,\rho) e_i)  \leq C3^{-2\alpha n}.
			\end{equation}
			Upon combining~\eqref{e.UpperBound.Corrector.Diff} and~\eqref{e.UpperBound.J}, the claim \eqref{eq.JBound} follows.
		\end{proof}
		
		\subsection{Semigroup and diffusion on configuration space}\label{subsec.Dirichletform}
		In this part, we summarize some properties of the semigroup and the associated diffusion process on the configuration space $\mmd(\Rd)$. One possible route to the definition of these objects would be to start by considering initial configurations with a finite number of particles. If $N$ particles are involved, the semigroup and stochastic process we aim to define are those associated with a divergence-form operator on $(\R^d)^N$, and standard methods apply: one can for instance construct the transition probabilities as in \cite[Appendix~E]{AKMbook} and then construct the stochastic process using the Kolmogorov extension and continuity theorems. Once this is done, one can then seek to define a stochastic process for sufficiently ``spread out'' initial point clouds (in particular for $\P_\rho$-almost every configuration) through a limit procedure. This is the method used in \cite{lang1977unendlich} for a slightly different model. 
		
		Since the stochastic process will not play a major role in this paper, we prefer to explain a more direct approach using the theory of Dirichlet forms. This approach has been explored in detail in \cite{albeverio1996differential, albeverio1996canonical, albeverio1998analysis, albeverio1998analysis2, kondratiev2003heat, ma2000construction, osada1996dirichlet, rockner1998stochastic, yoshida1996construction}, and the necessary conditions for the definition of the semigroup and stochastic process are easy to check for our model. 
		
		We endow the space $\mmd(\Rd)$ with the topology of vague convergence, and define the symmetric form on $\cL^2$ ($= L^2(\mmd(\Rd), \mcl F, \Pr)$) by setting
		\begin{align}\label{eq.Dirichlet}
			\Ll\{\begin{aligned}
				& \D(\Ee^{\a}) := \cH^1(\Rd),\\
				& \displaystyle{\forall u,v \in \D(\Ee^{\a}), \quad 			\Ee^{\a}(u,v) := \Er\Ll[\int_{\Rd} \frac{1}{2}\nabla u \cdot \a \nabla v \, \d \mu \Rr].}
			\end{aligned}\Rr.
		\end{align}
		The symmetric form $(\Ee^\a, \D(\Ee^\a))$ is closed: this means that if a sequence of elements of~$\D(\Ee^\a)$ is a Cauchy sequence with respect to the norm induced by the symmetric form $\Ee^{\a}_1(u,v) = \Ee^{\a}(u,v) + \bracket{u, v}_{\cL^2}$, then this sequence converges to an element of~$\D(\Ee^\a)$ with respect to this norm (see also \cite[Section~1.1]{fukushima2010dirichlet}). We can also verify that $\Ee^\a$ is Markovian, since for the unit contraction function $\phi_{\epsilon} \in C^\infty(\R)$ as in \cite[(1.1.5), Exercise~1.2.1]{fukushima2010dirichlet}, we have that $\vert \phi'_{\epsilon} \vert \leq 1$ and 
		\begin{align*}
			\forall u \in \D(\Ee^{\a}), \quad \Ee^\a(\phi_{\epsilon}(u), \phi_{\epsilon}(u)) = \Er\Ll[\int_{\Rd} \frac{1}{2}\vert \phi'_{\epsilon}(u)\vert^2 \nabla u \cdot \a \nabla u \, \d \mu \Rr] \leq \Ee^\a(u, u).
		\end{align*} 
		Therefore, $(\Ee^\a, \D(\Ee^\a))$ is a Dirichlet form. By the Riesz representation theorem (see \cite[Theorem~1.3.1]{fukushima2010dirichlet}), the Dirichlet form $(\Ee^\a, \D(\Ee^\a))$ defines a unique self-adjoint positive semi-definite operator $-A$ on $\cL^2$ such that 
		\begin{align}\label{eq.defA}
			\Ll\{\begin{array}{ll}
				\D(\Ee^{\a}) & = \D(\sqrt{-A}), \\
				\Ee^{\a}(u,v) & = \bracket{\sqrt{-A}u, \sqrt{-A}v}_{\cL^2},
			\end{array}\Rr.
		\end{align}
		where $\D(\sqrt{-A})$ denotes the domain of the operator $\sqrt{-A}$. Using the spectral decomposition, the self-adjoint operator $-A$ then defines a strongly continuous semigroup $P_t := \exp(tA)$ with $t \geq 0$ on $\cH^1(\Rd)$ (see \cite[Lemma~1.3.2]{fukushima2010dirichlet}). We will also exploit the close relationship between the semigroup and the resolvent of $-A$, which is the family of operators $(\lambda - A)^{-1}$ for $\lambda$ ranging in $(0,+\infty)$, and we refer to \cite[Section~1.3]{fukushima2010dirichlet} for more on this. We may use the informal notation $\frac{1}{2} \nabla \cdot \a \nabla$ to denote the operator $A$. 
		
		Since $\mmd(\Rd)$ is not locally compact with respect to the vague topology, we construct the associated Markov process using the notion of quasi-regular Dirichlet forms; see \cite[Chapter~IV.3]{ma2012introduction}, \cite{ma2000construction}, and here we follow the steps in \cite[Section~4]{rockner1998stochastic}. We denote by $\Ee$ the Dirichlet form with identity diffusion matrix
		\begin{align*}
			\Ll\{\begin{aligned}
				& \D(\Ee) := \cH^1(\Rd),\\
				& \displaystyle{\forall u,v \in \D(\Ee), \quad 			\Ee(u,v) := \Er\Ll[\int_{\Rd} \frac{1}{2}\nabla u \cdot  \nabla v \, \d \mu \Rr].}
			\end{aligned}\Rr.
		\end{align*}
		Our Dirichlet form $\Ee^{\a}$ is comparable to $\Ee$ because of the uniform ellipticity in \eqref{e.ass.unif.ell}, 
		\begin{align}\label{eq.Compare}
			\forall u \in \D(\Ee) = \D(\Ee^\a), \qquad \Ee(u,u) \leq \Ee^\a(u,u) \leq \Lambda \Ee(u,u).
		\end{align} 
		The quasi-regularity of $\Ee$ is proved in \cite[Theorem~4.12]{rockner1998stochastic} using the result in \cite{ma2000construction}. With the comparison in \eqref{eq.Compare}, the notions of $\Ee$-nest, $\Ee$-exceptional, and $\Ee$-quasi-continuous in \cite[Definitions~4.10 and~4.11]{rockner1998stochastic} are equivalent to those of $\Ee^\a$-nest, $\Ee^\a$-exceptional, and $\Ee^\a$-quasi-continuous, respectively. From this, we obtain  the quasi-regularity of $\Ee^\a$. The same argument also proves the locality of $\Ee^\a$ viewing that of $\Ee$ proved in \cite[Corollary~4.13]{rockner1998stochastic}. As a consequence of \cite[Theorems~IV.3.5, V.1.5 and~V.1.11]{ma2012introduction}, there exists an  $\mmd(\Rd)$-valued Markov process $(\bmu_t)_{t \geq 0}$ that is canonically associated with the Dirichlet form $\Ee^\a$, and we denote by $\PP_{\bmu_0}$ and $\EE_{\bmu_0}$ the probability and expectation associated with this process starting from the configuration $\bmu_0$. For every $u \in \cL^2$ and $\Ee^\a$-quasi every $\bmu_0 \in \mmd(\Rd)$, we have that 
		\begin{align*}
			(P_t u)(\bmu_0) = \EE_{\bmu_0}[u(\bmu_t)].
		\end{align*}
		Moreover, $\Pr$ is a reversible measure for  $(\bmu_t)_{t \geq 0}$.

		In this paper, we mostly focus on the properties of the semigroup. 
		\begin{proposition}[Elementary properties of the semigroup]\label{prop.ElementSemigroup}
			For every $u \in \cL^2$, the function $u_t := P_t u \in \cH^1(\Rd)$, $t > 0$, satisfy the following properties.
			\begin{enumerate}
				\item (Energy solution) For every $T > 0$, the function $u_\cdot (\cdot) : [0,T] \times \mmd(\Rd) \to \R$ is the unique solution in $L^2([0,T], \cH^1(\Rd)) \cap C([0,T], \cL^2)$ with $\partial_t u_t \in L^2([0,T], \cH^{-1}(\Rd))$ such that for every $t \in [0,T]$,
				\begin{align}\label{eq.semigroupA}
					\forall v \in \cH^1(\Rd), \qquad \Er[u_t v] - \Er[u v] + \int_{0}^t \Ee^\a(u_s, v) \, \d s = 0.
				\end{align}	
				The function $u_\cdot(\cdot)$ also satisfies the energy evolution identity
				\begin{align}\label{eq.L2Energy}
					\forall \tau \in [0, t), \qquad  \frac{1}{2}\norm{u_\tau}^2_{\cL^2} = \frac{1}{2}\norm{u_t}^2_{\cL^2} + \int_{\tau}^t  \Ee^\a (u_s, u_s) \, \d s.
				\end{align}
				\item (Decay of $\cL^2$ and energy) The mappings $t \mapsto \Er[(u_t)^2]$ and $t \mapsto \Ee^{\a}(u_t, u_t)$ are decreasing. Moreover, we have
				\begin{align}\label{eq.L2Decay}
					\forall \tau \in [0, t), \qquad \Ee^\a(u_t, u_t)  \leq \frac{1}{2(t - \tau)}\Er[(u_\tau)^2].    
				\end{align}
				Finally, when $u \in \D(\Ee^\a)$, we have $\Ee^{\a}(u_t, u_t) \xrightarrow{t \to 0} \Ee^{\a}(u, u)$.
			\end{enumerate}
		\end{proposition}
		\begin{proof}
			Following the classical energy solution theory of the evolution equation, there exists a unique solution $\tilde u_t \in L^2([0,T], \cH^1(\Rd)) \cap C([0,T], \cL^2)$ and $\partial_t \tilde u_t \in L^2([0,T], \cH^{-1}(\Rd))$; see \cite[Sections~7.1.2.b and 7.1.2.c, Theorems~2, 3, 4]{evans2022partial}. Moreover, using the definition \eqref{eq.defA}, $u_t = P_t u = e^{tA} u$ satisfies that for every $v \in \cH^1(\Rd)$ 
			\begin{align*}
				\bracket{u_t, v}_{\cL^2} - \bracket{u, v}_{\cL^2} &= \bracket{e^{tA} u, v}_{\cL^2} - \bracket{u, v}_{\cL^2} \\
				&= \int_{0}^t \bracket{A e^{sA}u, v}_{\cL^2} \, \d s \\
				&= -\int_{0}^t \bracket{\sqrt{-A} e^{sA}u, \sqrt{-A} v}_{\cL^2} \, \d s \\
				&= -\int_{0}^t \Ee^\a(u_s, v) \, \d s.
			\end{align*}
			Then the uniqueness of the solution in \eqref{eq.semigroupA} implies that $u_t = \tilde u_t$.  Equation \eqref{eq.L2Energy} follows from \eqref{eq.semigroupA} by testing with $v = u_t$, and it also implies that  $t \mapsto \Er[(u_t)^2]$ is decreasing.

			To prove that $t \mapsto \Ee^{\a}(u_t, u_t)$ is also decreasing, we use the spectral measure decomposition for the Dirichlet form
			\begin{align}\label{eq.Spectral}
				\Ee^{\a}(u_t, u_t) = \bracket{e^{tA} u, - A e^{tA} u}_{\cL^2} = \int_{\lambda = 0}^\infty \lambda e^{-2 \lambda t} \,  \d \bracket{E_\lambda u, u}_{\cL^2}. 
			\end{align}
			Here $E_{\lambda}$ is the spectral projection operator with respect to $-A$; see \cite[Lemma~1.3.2]{fukushima2010dirichlet}. We see that the quantity above is decreasing with respect to $t$, as announced. We put this property into \eqref{eq.L2Energy}, which gives 
			\begin{align*}
				(t-\tau)\Ee^\a (u_t, u_t) \leq \int_{\tau}^t  \Ee^\a (u_s, u_s) \, \d s \leq  \frac{1}{2}\norm{u_\tau}^2_{\cL^2}. 
			\end{align*}
			This shows \eqref{eq.L2Decay}. When $u \in \D(\Ee^\a)$, we have 
			\begin{align*}
				\Ee^{\a}(u, u) = \bracket{\sqrt{- A} u,  \sqrt{- A} u}_{\cL^2} = \int_{\lambda = 0}^\infty \lambda \,  \d \bracket{E_\lambda u, u}_{\cL^2} < + \infty. 
			\end{align*}
			Combining this with an application of the dominated convergence theorem in  \eqref{eq.Spectral}, we obtain that $\lim_{t \to 0}\Ee^{\a}(u_t, u_t)= \Ee^{\a}(u, u)$.
		\end{proof}

		\section{Two-scale expansion for elliptic equation}
		\label{sec:3}
		In this section, we take the Dirichlet problem of elliptic equations as a simple example to explain the two-scale expansion technique on the configuration space. Recall the homogenized matrix $\ab$ defined in \eqref{eq:AbarDef}, where we keep the dependence on $\rho$ implicit. Consider $U, \Ub \in \cH^1_0(\cu_m)$ centered solving
		\begin{align}\label{eq.defDirichlet}
			-\nabla \cdot (\a \nabla U) = {F}, \qquad -\nabla \cdot (\ab \nabla \Ub) =  F,
		\end{align}
		in the sense that for all $V \in \cH^1_0(\cu_m)$,
		\begin{align}
			\Er\Ll[\int_{\cu_m} \nabla V \cdot \a \nabla U \, \d \mu \Rr] &= \Er[V F],  \label{eq.Dirichlet1}\\
			\Er\Ll[\int_{\cu_m} \nabla V \cdot \ab \nabla \Ub \, \d \mu \Rr] &= \Er[V F].     \label{eq.Dirichlet2}       
		\end{align}
		We study the homogenization problem in the case when $F$ is a centered linear statistic, meaning that
		\begin{align}\label{eq.defF}
			F(\mu) := \int_{\cu_m} f \, \d \mu,  \qquad \text{ with } \qquad \int_{\cu_m} f \, \d m = 0.
		\end{align}
		Here $f \in L^2(\cu_m)$ is a function on Euclidean space $\Rd$ rather than the configuration space. This particular case is convenient because its homogenized solution has an explicit expression related to the homogenized solution in Euclidean space.
		
		\begin{lemma}
			Under the condition \eqref{eq.defF}, the unique solution of \eqref{eq.Dirichlet2} such that $\Er[\Ub] = 0$ is
			\begin{align}\label{eq.defUb}
				\Ub(\mu) = \int_{\cu_m} \ub \, \d \mu - \rho \int_{\cu_m} \ub \, \d m, 
			\end{align}
			with $\ub \in H^1_0(\cu_m)$ the unique solution of 
			\begin{equation}\label{eq.defub}
				\Ll\{\begin{array}{ll}
					- \nabla \cdot (\ab \nabla \ub) = f, & \text{ in } \cu_m, \\
					\ub = 0, & \text{ on } \partial \cu_m. 
				\end{array}\Rr. 
			\end{equation}
		\end{lemma}
		\begin{proof} It is clear that $\Ub \in \cH^1_0(\cu_m)$ as $\ub \in H^1_0(\cu_m)$. If we test \eqref{eq.Dirichlet2} with some $V \in \cH^1_0(\cu_m)$, then its {\lhs} is 
			\begin{multline}\label{eq.IPPLHS}
				\Er\Ll[\int_{\cu_m} \nabla V \cdot \ab \nabla \Ub \, \d \mu \Rr] \\
				= e^{-\rho \vert \cu_m \vert} \sum_{k=1}^{\infty} \frac{(\rho \vert \cu_m \vert)^k}{k!} \fint_{(\cu_m)^k} \sum_{i=1}^k \nabla_{x_i} \tilde{V}_k(x_1, \cdots, x_k) \cdot \ab \nabla \ub(x_i) \, \d x_1 \cdots \d x_k,
			\end{multline}
			where $\tilde{V}_k(x_1, \cdots, x_k) := V(\sum_{i=1}^k \delta_{x_i})$ is the projection of the function conditioned on the number of particles. We treat the integration with respect to $x_i$ for each term 
			\begin{equation}\label{eq.IPP}
				\begin{split}
					&\int_{\cu_m}  \nabla_{x_i} \tilde{V}_k(x_1, \cdots, x_k) \cdot \ab \nabla \ub(x_i) \, \d x_i \\
					&= \int_{\cu_m}   \nabla_{x_i} \Ll(\tilde{V}_k(x_1, \cdots, x_k) - \tilde{V}_{k-1}(x_1, \cdots, x_{i-1},  x_{i+1}, \cdots, x_k) \Rr) \cdot \ab \nabla \ub(x_i)  \, \d x_i \\
					&=  \int_{\cu_m}   \Ll(\tilde{V}_k(x_1, \cdots, x_k) - \tilde{V}_{k-1}(x_1, \cdots, x_{i-1},  x_{i+1}, \cdots, x_k) \Rr) f(x_i)  \, \d x_i \\
					&=  \int_{\cu_m}   \tilde{V}_k(x_1, \cdots, x_k)  f(x_i)  \, \d x_i. \\
				\end{split}
			\end{equation}	
			The equality between the expressions in the second and third lines is valid by the definition of $\bar u$ in \eqref{eq.defub} and since the mapping
			\begin{align*}
				x_i \mapsto \Ll(\tilde{V}_k(x_1, \cdots, x_k) - \tilde{V}_{k-1}(x_1, \cdots, x_{i-1},  x_{i+1}, \cdots, x_k) \Rr)
			\end{align*}
			is in $H^1_0(\cu_m)$ (see for instance \cite[(3.8)]{giunti2021smoothness}). From the third line to the fourth line, we make use of the property $\int_{\cu_m} f \, \d m = 0$ in \eqref{eq.defF} and $\tilde{V}_{k-1}(x_1, \cdots, x_{i-1},  x_{i+1}, \cdots, x_k)$ constant for $x_i \in \cu_m$. We put this back into \eqref{eq.IPPLHS} and obtain that 
			\begin{align*}
				&\Er\Ll[\int_{\cu_m} \nabla V \cdot \ab \nabla \Ub \, \d \mu \Rr] \\
				&= e^{-\rho \vert \cu_m \vert} \sum_{k=1}^{\infty} \frac{(\rho \vert \cu_m \vert)^k}{k!} \fint_{(\cu_m)^k} \tilde{V}_k(x_1, \cdots, x_k) \sum_{i=1}^k f(x_i) \, \d x_1 \cdots \d x_k \\
				&= \Er[V F]. 
			\end{align*}
			Here we remark that the case $\mu(\cu_m) = 0$ does not contribute in the definition of $F$ in~\eqref{eq.defF}. This finishes the proof.
		\end{proof}
		
		\bigskip

		Now, with the explicit expression \eqref{eq.defUb}, we propose the following two-scale expansion 
		\begin{align}\label{eq.defW}
			W := \Ub + \sum_{i=1}^d \sum_{z \in \Z_{m,n}} (\partial_i \ub)_{z+\cu_n} \phi_{e_i, z+\cu_n},
		\end{align}
		to approximate both $U$ and $\Ub$, i.e. $W \simeq U \simeq \Ub$ in $\cL^2$, and also with $W \simeq U$ in $\cH^1(\cu_m)$. The integer $n \le m$ will be chosen shortly depending on $m$. We recall that $\Z_{m,n} = 3^n \Zd \cap \cu_m$ as in~\eqref{e.Z.def} and $\phi_{e_i, z+\cu_n}$ is the $\cH^1_0(z+\cu_n)$ corrector defined in \eqref{eq.corrector}, and we define $(\partial_i \ub)_{z+\cu_n}$ as 
		\begin{align*}
			(\partial_i \ub)_{z+\cu_n} := \fint_{z+\cu_n} \partial_i \ub,
		\end{align*}
		which is a deterministic constant. Recall that $\ell_{e_i, z+\cu_n} + \phi_{e_i, z+\cu_n}$ is $\a$-harmonic on the domain $z+\cu_n$ by~\eqref{eq:VariationalFormula72} and~\eqref{eq.corrector}. We use $(\partial_i \ub)_{z+\cu_n}$ to represent the slope in the mesoscopic domain $z+\cu_n$, and use $ (\partial_i \ub)_{z+\cu_n} \phi_{e_i, z+\cu_n}$  to correct the solution in $z+\cu_n$. The following result can be seen as a generalization of two-scale expansion on the configuration space. Notice that the $\cH^1$ estimate and Poincar\'e inequality on \eqref{eq.defDirichlet} imply that 
		\begin{align}\label{eq.ClassicalH1}
			\norm{U}_{\acH^1(\cu_m)} \leq C 3^m \norm{F}_{\cL^2}, \qquad \norm{\Ub}_{\acH^1(\cu_m)} \leq C 3^m \norm{F}_{\cL^2},
		\end{align}
		so the error term in the following result is smaller than the bound above on large scales.
		
		\begin{proposition}[Two-scale expansion of elliptic equation]\label{prop.TwoScaleElliptic}
			The following holds for the exponent $\alpha(d,\lambda,\rho) > 0$  given by Lemma~\ref{lem.JBound} and for some constant $C(d,\lambda,\rho) < +\infty$. 
			Under the condition \eqref{eq.defF}, with the two-scale expansion \eqref{eq.defW} associated with \eqref{eq.Dirichlet1}, \eqref{eq.Dirichlet2} and with the choice of $n := \lfloor \frac {m} {1+\alpha} \rfloor$, we have
			\begin{align}\label{eq.TwoScaleElliptic}
				\norm{W - U}_{\acH^1(\cu_m)} + 3^{-m} \Ll(\norm{W-\Ub}_{\cL^2} + \norm{U-\Ub}_{\cL^2}\Rr)	\leq C 3^{\frac{m}{1+\alpha}} \norm{F}_{\cL^2}.
			\end{align}
		\end{proposition}
		\begin{proof} 			
			\textit{Step~1: estimate of $\nabla (W - U)$.} The main step of the proposition is to study $\nabla (W - U)$, because the two-scale expansion provides $\cH^1$ approximation for $U$. Using \eqref{eq.defDirichlet}, we take any $V \in \cH^1_0(\cu_m)$ and study 
			\begin{equation}\label{eq.ellipticDecom}
				\begin{split}
					\Er\Ll[\int_{\cu_m} \nabla V \cdot \a \nabla (W-U) \, \d \mu\Rr] &=\Er\Ll[\int_{\cu_m} \nabla V \cdot (\a \nabla W - \ab \nabla \Ub) \, \d \mu\Rr] \\
					&= \mathbf{I} + \mathbf{II},
				\end{split}
			\end{equation}
			which we decompose as the sum of the following terms 
			\begin{align*}
				\mathbf{I} &:= \sum_{z \in \Z_{m,n}}  \Er\Ll[\int_{z+\cu_n} \nabla V \cdot (\a-\ab)(\nabla \ub - (\nabla \ub)_{z+\cu_n}) \, \d \mu\Rr], \\
				\mathbf{II} &:= \sum_{i = 1}^d \sum_{z \in \Z_{m,n}}  (\partial_i \ub)_{z+\cu_n} \Er\Ll[\int_{z+\cu_n} \nabla V \cdot \Ll(\a(e_i + \nabla \phi_{e_i, z+\cu_n}) - \ab e_i\Rr) \, \d \mu\Rr].
			\end{align*}
			This decomposition is even simpler than the classical two-scale expansion in $\Rd$, because here the factor $(\partial_i \ub)_{z+\cu_n}$ is totally deterministic and no derivatives act on it. On the other hand, the corrector $\phi_{e_i, z+\cu_n}$ already contains all the influence from the interaction of particles in $z+\cu_n$. Here the term $\mathbf{I}$ can be seen as the error to fix the slope in the mesoscopic scale $3^n$, while the term $\mathbf{II}$ is the error of the flux replacement. We estimate them separately.
			
			\smallskip
			\textit{Step~2.1: term $\mathbf{I}$ as the error to fix the slope.} For the term $\mathbf{I}$, we apply the Cauchy--Schwarz inequality to get that
			\begin{align*}
				\vert \mathbf{I} \vert &\leq \Lambda\norm{V}_{\acH^1(\cu_m)}  \Ll(\sum_{z \in \Z_{m,n}}\Er\Ll[\int_{z+\cu_n} (\nabla \ub - (\nabla \ub)_{z+\cu_n})^2 \, \d \mu\Rr]\Rr)^{\frac{1}{2}}.
			\end{align*}
			Using Poincar\'e's inequality, we have
			\begin{equation}\label{eq.TermI}
				\begin{split}
					\sum_{z \in \Z_{m,n}}\Er\Ll[\int_{z+\cu_n} (\nabla \ub - (\nabla \ub)_{z+\cu_n})^2 \, \d \mu\Rr] &\leq  C 3^{2n} \rho \sum_{z \in \Z_{m,n}} \int_{z+\cu_n} \vert\nabla \nabla \ub\vert^2 \, \d m  \\
					&= C 3^{2n} \rho \int_{\cu_m} \vert\nabla \nabla \ub\vert^2 \, \d m.
				\end{split}
			\end{equation}

			\smallskip
			\textit{Step~2.2: term $\mathbf{II}$ as the error of flux replacement.} For the term $\mathbf{II}$, we would like to apply the flux replacement as the slope is already fixed by $(\partial_i \ub)_{z+\cu_n}$. However, as $V$ is not a local function belonging to $\cH^1_0(z+\cu_n)$, some manipulation is required to test against it. We replace $\phi_{e_i, z+\cu_n}$ by the dual corrector $\phi^*_{e_i, z+\cu_n}$ associated to the dual quantity $\nu^*(z+\cu, \ab_*(z+\cu_n) e_i)$, since $\phi^*_{e_i, z+\cu_n}$ admits more test functions from \eqref{eq:VariationalFormula72}.
			\begin{align*}
				\mathbf{II} &= \mathbf{II.1} + \mathbf{II.2},\\
				\mathbf{II.1} &:= \sum_{i = 1}^d \sum_{z \in \Z_{m,n}}  (\partial_i \ub)_{z+\cu_n} \Er\Ll[\int_{z+\cu_n} \nabla V \cdot \a\Ll( \nabla \phi_{e_i, z+\cu_n} - \nabla \phi^*_{e_i, z+\cu_n}\Rr) \, \d \mu\Rr],\\
				\mathbf{II.2} &:= \sum_{i = 1}^d \sum_{z \in \Z_{m,n}}  (\partial_i \ub)_{z+\cu_n} \Er\Ll[\int_{z+\cu_n} \nabla V \cdot \Ll(\a(e_i + \nabla \phi^*_{e_i, z+\cu_n}) - \ab e_i\Rr) \, \d \mu\Rr].
			\end{align*}
			For the term $\mathbf{II.1}$, we use the Cauchy--Schwarz inequality, which gives us
			\begin{equation}\label{eq.TermII1}
				\begin{split}
					\vert \mathbf{II.1}\vert &\leq C \norm{V}_{\acH^1(\cu_m)}\Ll(\sum_{i = 1}^d \sum_{z \in \Z_{m,n}}  \vert (\partial_i \ub)_{z+\cu_n} \vert^2  \Er\Ll[\int_{z+\cu_n} \vert \nabla \phi_{e_i, z+\cu_n} - \nabla \phi^*_{e_i, z+\cu_n} \vert^2 \, \d \mu\Rr]\Rr)^{\frac{1}{2}}\\
					&\leq C 3^{-\alpha n} \norm{V}_{\acH^1(\cu_m)}\Ll(\sum_{i = 1}^d \sum_{z \in \Z_{m,n}}  \rho  \vert \cu_n \vert\vert (\partial_i \ub)_{z+\cu_n} \vert^2 \Rr)^{\frac{1}{2}}\\
					&\leq C 3^{-\alpha n} \norm{V}_{\acH^1(\cu_m)} \norm{\Ub}_{\acH^1(\cu_m)}
				\end{split}
			\end{equation}
			Here the error from $\nabla (\phi_{e_i, z+\cu_n} - \phi^*_{e_i, z+\cu_n})$ 
			could be bounded by employing Lemma~\ref{lem.JBound}.

			For the term $\mathbf{II.2}$, we use the variational formula~\eqref{eq:VariationalFormula72} for the maximizer of $\nu^*(z+\cu, \ab_*(z+\cu_n) e_i)$ to obtain
			\begin{align*}
				\Er\Ll[\int_{z+\cu_n} \nabla V \cdot \a (e_i + \nabla \phi^*_{e_i, z+\cu_n}) \, \d \mu\Rr] = \Er\Ll[\int_{z+\cu_n} \nabla V \cdot \ab_*(z+\cu_n) e_i  \, \d \mu\Rr].
			\end{align*}
			Then we obtain that 
			\begin{equation}\label{eq.TermII2}
				\begin{split}
					\vert \mathbf{II.2}\vert &= \Ll\vert\sum_{i = 1}^d \sum_{z \in \Z_{m,n}}  (\partial_i \ub)_{z+\cu_n} \Er\Ll[\int_{z+\cu_n} \nabla V \cdot (\ab_*(z+\cu_n) - \ab) e_i \, \d \mu\Rr]\Rr\vert\\
					&\leq C 3^{-\alpha n} \norm{V}_{\acH^1(\cu_m)} \norm{\Ub}_{\acH^1(\cu_m)},
				\end{split}
			\end{equation}
			because \cite[Theorem~2.1]{bulk} implies that $\vert \ab_*(z+\cu_n) - \ab \vert \leq C3^{-\alpha n}$, as recalled in \eqref{eq:QuantConvergence}.
			
			\smallskip
			\textit{Step~2.3: choice of parameters.}
			Combining the estimates \eqref{eq.TermI}, \eqref{eq.TermII1}, and \eqref{eq.TermII2}, we obtain 
			\begin{multline}\label{eq.TwoScaleMixed}
				\Ll\vert \Er\Ll[\int_{\cu_m} \nabla V \cdot \a \nabla (W-U) \, \d \mu\Rr]\Rr\vert  \\
				\leq C \norm{V}_{\acH^1(\cu_m)} \Ll( \rho  \int_{\cu_m} 3^{2n} \vert\nabla \nabla \ub\vert^2  + 3^{-2\alpha n} \vert \nabla \ub\vert^2 \, \d m \Rr)^\frac{1}{2}.
			\end{multline}
			In order to avoid confusion, here we write down all the integration with respect to $\ub$ explicitly. Then we apply $H^1$ and $H^2$ estimates for $\ub$ defined in \eqref{eq.defub}
			\begin{align*}
				\rho  \int_{\cu_m} 3^{2n} \vert\nabla \nabla \ub\vert^2  + 3^{-2\alpha n} \vert \nabla \ub\vert^2 \, \d m &\leq  C (3^{2n} + 3^{2m-2\alpha n}) \rho  \int_{\cu_m} f^2 \, \d m\\
				&= C (3^{2n} + 3^{2m-2\alpha n}) \norm{F}^2_{\cL^2}.
			\end{align*} 
			The validity of the $H^2$ estimate for $\ub$ is asserted for general convex domains in \cite[Theorem~3.1.2.1]{grisvard1985elliptic} (notice that the constant $C(\Omega)$ appearing there only depends on the diameter of $\Omega$); in the particular case of cubes we can also more directly adapt the proof of \cite[Lemma~B.19]{AKMbook} to the case of a Dirichlet boundary condition (the extension of the function $f$ appearing in this proof can be taken as $f(x_1,\ldots, x_d) := \mathrm{sgn}(x_1) \cdots \mathrm{sgn}(x_d) f(|x_1|,\ldots, |x_d|)$ in this case). 
			With the choice of $n = \lfloor \frac{m}{1+\alpha} \rfloor$, we obtain that 
			\begin{align*}
				\Ll\vert \Er\Ll[\int_{\cu_m} \nabla V \cdot \a \nabla (W-U) \, \d \mu\Rr] \Rr\vert \leq C 3^{\frac{m}{1+\alpha}} \norm{V}_{\acH^1}(\cu_m) \norm{F}_{\cL^2(\cu_m)}. 
			\end{align*}
			Viewing $W$ as the sum of $\Ub$ and a linear combination of some $\cH^1_0(\cu_m)$ functions,  it also belongs to $\cH^1_0(\cu_m)$. We choose $V = W-U \in \cH^1_0(\cu_m)$ and use the Poincar\'e inequality \eqref{e.Poincare} and the uniform ellipticity condition in \eqref{e.ass.unif.ell} to obtain that
			\begin{align}\label{eq.TwoScaleH-1}
				\|W-U\|_{\acH^1(\cu_m)} 
				\leq C 3^{\frac{m}{1+\alpha}} \norm{F}_{\cL^2(\cu_m)}. 
			\end{align}

			\smallskip
			\textit{Step 3: estimate of $\norm{W-\Ub}_{\cL^2}$.} The $L^2$  estimate of $\norm{W-\Ub}^2_{\cL^2}$  can be done directly
			\begin{align*}
				\norm{W-\Ub}^2_{\cL^2} &=  \Er\Ll[\Ll(\sum_{i=1}^d \sum_{z \in \Z_{m,n}} (\partial_i \ub)_{z+\cu_n} \phi_{e_i, z+\cu_n}\Rr)^2\Rr]\\
				&\leq  C \sum_{i=1}^d \sum_{z \in \Z_{m,n}} (\partial_i \ub)^2_{z+\cu_n} \Er\Ll[ \phi^2_{e_i, z+\cu_n}\Rr] \\
				&\leq  C 3^{2n} \sum_{i=1}^d \sum_{z \in \Z_{m,n}} (\partial_i \ub)^2_{z+\cu_n}  \Er\Ll[ \int_{z+\cu_n} \vert \nabla \phi_{e_i, z+\cu_n}\vert^2 \, \d \mu\Rr]\\
				&\leq C 3^{2n} \norm{\Ub}^2_{\acH^1(\cu_m)}.
			\end{align*}
			Here in the second line we use the fact that $\phi_{e_i, z+\cu_n}$ is local function,  thus for $\dist(z, z') > 3^n$, the associated correctors are independent; see (1) of Lemma~\ref{lem.JBound}. Then we use the basic $\cH^1$ estimate of $\Ub$ and Poincar\'e's inequality to conclude
			\begin{align}\label{eq.TwoScaleL2}
				3^{-m} \norm{W-\Ub}_{\cL^2} \leq C 3^{n} \norm{F}_{\cL^2}.
			\end{align}
			Recalling the choice of $n = \lfloor \frac{m}{1+\alpha} \rfloor$, we obtain the desired estimate.
		\end{proof}
		\begin{remark}
			One can also study the homogenization of elliptic equation \eqref{eq.defDirichlet} without the assumption \eqref{eq.defF}, and this will be an interesting question for future work. 
		\end{remark}
		
		\section{Quantitative estimate of two-point function and semigroup}
		\label{sec:4}
		In this section, we use a parabolic two-scale expansion to prove Theorem~\ref{t.two.point}. Recall that $P_t$ is the parabolic semigroup associated with $(\Ee^\a, \D(\Ee^\a))$ defined in \eqref{eq.defA}, and let $\Pb_t$ be the one associated to $(\Ee^\ab, \D(\Ee^\ab))$ defined similarly. Theorem~\ref{t.two.point} is also related to the homogenization of the parabolic semigroup.
		\begin{proposition}\label{prop.HomoSemigroup}
			There exists an exponent $\beta(d,\Lambda,\rho) > 0$ and  $C(d,\Lambda,\rho) < + \infty$ such that for every $f \in L^1(\Rd) \cap L^2(\Rd)$, its associated linear statistic $F(\mu) := \int_{\Rd} f \, \d \mu$ satisfies, for every $t > 0$, 
			\begin{align}\label{eq.HomoSemigroup}
				\norm{(P_t  - \Pb_t)F}_{\cL^2} \leq C t^{- \frac{\beta}{2}} \norm{f}_{L^2}.
			\end{align}
		\end{proposition}

		Our proof essentially follows the outline of \cite{zhikov2006estimates}. 
		In the proof, there are two main ingredients: the two-scale expansion and the regularization by the parabolic semigroup. 
		Since $(\Ee^\ab, \D(\Ee^\ab))$ defines the diffusion of independent Brownian motions, a very important observation is the explicit formula of $\Pb_t F(\mu)$ for the linear statistic $F(\mu) = \int_{\Rd} f \, \d \mu$, which reads as
		\begin{align}\label{eq.PbHomo}
			\Pb_t F (\mu) = \int_{\Rd}  \Psi_t \star f \, \d \mu,
		\end{align} 
		where we recall that $(\Psi_t(x))_{t \ge 0, x \in \Rd}$ is the heat kernel defined in \eqref{e.def.heat.kernel}, and $\star$ denotes the spatial convolution in $\Rd$. 
		We use the following classical heat kernel estimates on $\Rd$ (see e.g.\ \cite[(1.17)-(1.18)]{zhikov2006estimates}): for every $k \in \N$, there exists a constant $C_k(d, \Lambda, \rho) < +\infty$ such that for every $g \in L^2$ and with $g_t := \Psi_t \star g$,
		\begin{align}\label{eq.heatkernelClassic}
			\forall 0 < \tau < t < +\infty,  \qquad  \int_\tau^t \norm{\nabla^k g_s}^2_{L^2} \, \d s &\leq C_k \norm{\nabla^{k-1} g_\tau}^2_{L^2},\\
			\nonumber \norm{\nabla^k g_t}_{L^2} &\leq \frac{ C_k}{(t-\tau)^{k/2}}\norm{g_\tau}_{L^2}.
		\end{align}

		\subsection{Two-scale expansion of parabolic equation}
		We first give a proof of homogenization in the special case of a rather smooth initial condition, using a two-scale expansion. For any fixed $g \in L^1(\Rd) \cap L^2(\Rd)$, we denote its linear statistic by $G := \int_\Rd g \, \d \mu - \rho \int_{\Rd} g \, \d m$, and $G_t := P_t G$, $\bar{G}_t := \Pb_t G$. We propose its associated two-scale expansion
		\begin{align}\label{eq.TwoScalePara}
			\tilde G_t := \bar{G}_t + \sum_{i = 1}^d \sum_{z \in \Z_n} (\partial_i g_t)_{z+\cu_n} \phi_{e_i, z+\cu_n},
		\end{align} 
		where $\Z_n$ is defined in~\eqref{e.Z.def} and $g_t := \Psi_t \star g$. Our first estimate also assumes that $g$ belongs to $H^2(\Rd)$. 
		\begin{proposition}\label{prop.TwoScalePara}
			Let $\alpha(d, \Lambda, \rho) > 0$ be as given by Lemma~\ref{lem.JBound}. There exists a constant $C(d, \Lambda, \rho) < +\infty$ such that for every $g \in H^2(\Rd) \cap L^1(\Rd)$, $n \in \N_+$, and $t >0$, its linear statistic and associated two-scale expansion defined above satisfy 
			\begin{multline}\label{eq.TwoScaleParaBound}
				\norm{\tilde G_t - \bar G_t}_{\cL^2} + \norm{\tilde G_t - G_t}_{\cL^2} + \Ll(\int_0^t \Er\Ll[\int_{\Rd} \vert \nabla (\tilde G_s - G_s)\vert^2 \, \d \mu \Rr] \, \d s\Rr)^{\frac{1}{2}}\\
				\leq C \Ll(\Ll(3^{-\alpha n} + t^{-\frac{1}{8}}\Rr)\norm{g}_{L^2} + 3^n \norm{\nabla g}_{L^2} + 3^n t^{\frac{5}{8}}\norm{\nabla^2 g}_{L^2}\Rr).
			\end{multline}
		\end{proposition}
		\begin{proof}		
			The part $\norm{\tilde G_t - \bar G_t}_{\cL^2}$ can be bounded directly
			\begin{equation}\label{eq.TwoScaleParaBound1}
				\begin{split}
					\norm{\tilde G_t - \bar G_t}^2_{\cL^2} &=  \Er\Ll[\Ll(\sum_{i=1}^d \sum_{z \in \Z_n} (\partial_i g_t)_{z+\cu_n} \phi_{e_i, z+\cu_n}\Rr)^2\Rr]\\
					&\leq  C \sum_{i=1}^d \sum_{z \in \Z_n} (\partial_i g_t)^2_{z+\cu_n} \Er\Ll[ \phi^2_{e_i, z+\cu_n}\Rr] \\
					&\leq  C 3^{2n} \sum_{i=1}^d \sum_{z \in \Z_n} (\partial_i g_t)^2_{z+\cu_n}  \Er\Ll[ \int_{z+\cu_n} \vert \nabla \phi_{e_i, z+\cu_n}\vert^2 \, \d \mu\Rr]\\
					&\leq C 3^{2n} \norm{\nabla g_t}^2_{L^2}\\
					&\leq C 3^{2n} \norm{\nabla g}^2_{L^2}.
				\end{split}
			\end{equation}		
			In the first three lines, we use the independence of the approximate correctors in (1) of Lemma~\ref{lem.JBound}. In the last line, we use the contraction from the convolution $\Psi_t$.  
			
			The main part of the proof is the estimation of $(\tilde G_t - G_t)$. Using the parabolic equations associated to $G_s$ and $\bar G_s$, we have
			\begin{align*}
				\Ll(\partial_s - \frac{1}{2} \nabla \cdot \a \nabla\Rr) G_s = \Ll(\partial_s - \frac{1}{2} \nabla \cdot \ab \nabla\Rr) \bar G_s = 0,
			\end{align*}
			and thus
			\begin{align*}
				\Ll(\partial_s - \frac{1}{2} \nabla \cdot \a \nabla\Rr)(\tilde G_s - G_s) = \Ll(\partial_s - \frac{1}{2} \nabla \cdot \a \nabla\Rr)\tilde G_s - \Ll(\partial_s - \frac{1}{2} \nabla \cdot \ab \nabla\Rr) \bar G_s .
			\end{align*}
			We test this with $(\tilde G_s - G_s)$ and integrate over $[0,t]$ 
			\begin{equation}\label{eq.paraDecom}
				\begin{split}
					&\norm{\tilde G_t - G_t}^2_{\cL^2} + \int_0^t \Er\Ll[\int_{\Rd} \vert \nabla(\tilde G_s - G_s) \vert^2 \, \d \mu \Rr] \, \d s \\
					&\leq \norm{\tilde G_0 - G_0}^2_{\cL^2} \\
					& \qquad + \int_{0}^t \Er\Ll[\int_{\cu_m} \nabla (\tilde G_s - G_s) \cdot (\a \nabla \tilde G_s - \ab \nabla \bar G_s) \, \d \mu \Rr] \, \d s \\
					& \qquad + \int_{0}^t   \Er\Ll[ (\tilde G_s - G_s) \sum_{i=1}^d \sum_{z \in \Z_n} (\partial_i \partial_s g_s)_{z+\cu_n}\phi_{e_i, z+\cu_n}\Rr]\, \d s. 
				\end{split}
			\end{equation}
			It suffices to estimate the three terms on the {\rhs} to obtain \eqref{eq.TwoScaleParaBound}. The first term can be estimated similarly to $\norm{\tilde G_t - \bar G_t}^2_{\cL^2}$ in \eqref{eq.TwoScaleParaBound1} as the evolution has not started and $G_0  = \bar G_0$, so
			\begin{align}\label{eq.paraDecom1}
				\norm{\tilde G_0 - G_0}^2_{\cL^2} \leq C 3^{2n} \norm{\nabla g}^2_{L^2}.
			\end{align}
			The third term in \eqref{eq.paraDecom} only depends on the elementary $\cL^2$ bound, as $\partial_i \partial_s g_s$ gains enough decay from \eqref{eq.heatkernelClassic}
			\begin{equation}\label{eq.paraDecom3}
				\begin{split}
					&\Ll\vert \int_{0}^t   \Er\Ll[ (\tilde G_s - G_s) \sum_{i=1}^d \sum_{z \in \Z_n} (\partial_i \partial_s g_s)_{z+\cu_n}\phi_{e_i, z+\cu_n}\Rr]\, \d s \Rr\vert \\
					&\leq  \int_{0}^t \frac{t^{-\frac{5}{4}}}{2} \norm{\tilde G_s - G_s}^2_{\cL^2} +  \frac{t^{\frac{5}{4}}}{2} \norm{\sum_{i=1}^d \sum_{z \in \Z_n}(\partial_i \partial_s g_s)_{z+\cu_n}\phi_{e_i, z+\cu_n}}^2_{\cL^2} \, \d s \\
					&\le C  \int_{0}^t t^{-\frac{5}{4}} \Ll(\norm{\bar G_s}^2_{\cL^2} + \norm{G_s}^2_{\cL^2} + \sum_{z \in \Z_n} \norm{\sum_{i = 1}^d (\partial_i g_t)_{z+\cu_n} \phi_{e_i, z+\cu_n}}^2_{\cL^2}\Rr)  \\
					&  \qquad \qquad +  t^{\frac{5}{4}} \sum_{z \in \Z_n}\norm{\sum_{i=1}^d (\partial_i \partial_s g_s)_{z+\cu_n}\phi_{e_i, z+\cu_n}}^2_{\cL^2} \, \d s \\
					&\leq C\int_{0}^t t^{-\frac{5}{4}}  (\norm{g}^2_{L^2} + 3^{2n}\|\nabla g\|_{L^2}^2)  + 3^{2n} t^{\frac{5}{4}}  \norm{\nabla \partial_s g_s}^2_{L^2}\, \d s\\
					&\leq C \Ll(t^{-\frac{1}{4}} \norm{g}^2_{L^2} + t^{-\frac{1}{4}} 3^{2n}\|\nabla g\|_{L^2}^2 + 3^{2n} t^{\frac{5}{4}}  \norm{\nabla^2 g}^2_{L^2}\Rr).
				\end{split}
			\end{equation}
			Here in the second line, we apply Young's inequality  and choose the weight $t^{\frac{5}{4}}$ to better balance each term (it suffices to take $t^{\theta}$ with an exponent $\theta \in (1,2)$). In the third line, we apply the triangle inequality to $\norm{\tilde G_s - G_s}_{\cL^2}$ with the expression \eqref{eq.TwoScalePara} at first, then we apply the independence of the approximate correctors in Lemma~\ref{lem.JBound} to the term of type $\norm{\sum_{i=1}^d \sum_{z \in \Z_n}(\partial_i \partial_s g_s)_{z+\cu_n}\phi_{e_i, z+\cu_n}}^2_{\cL^2}$. From the fourth line to the fifth line, we use \eqref{eq.heatkernelClassic} and the identity $\partial_s g_s = \frac{1}{2} \nabla \cdot (\ab \nabla g_s)$.  
			
			Using the explicit expression of $\bar G_s$ in \eqref{eq.PbHomo}, the second term in \eqref{eq.paraDecom} can be treated like the elliptic case \eqref{eq.ellipticDecom} by the following decomposition
			\begin{align}\label{eq.paraDecom2Decom}
				\int_{0}^t \Er\Ll[\int_{\Rd} \nabla (\tilde G_s - G_s) \cdot (\a \nabla \tilde G_s - \ab \nabla \bar G_s) \, \d \mu \Rr] \, \d s = \mathbf{I'} + \mathbf{II'},
			\end{align}
			where the two terms are respectively
			\begin{align*}
				\mathbf{I'} & := \int_{0}^t \sum_{z \in \Z_{n}}  \Er\Ll[\int_{z+\cu_n} \nabla (\tilde G_s - G_s) \cdot (\a-\ab)(\nabla g_s - (\nabla g_s)_{z+\cu_n}) \, \d \mu\Rr] \, \d s, \\
				\mathbf{II'} & := \int_{0}^t  \sum_{i = 1}^d \sum_{z \in \Z_{n}}  (\partial_i g_s)_{z+\cu_n} \Er\Ll[\int_{z+\cu_n} \nabla (\tilde G_s - G_s) \cdot \Ll(\a(e_i + \nabla \phi_{e_i, z+\cu_n}) - \ab e_i\Rr) \, \d \mu\Rr] \, \d s.
			\end{align*} 
			For the term $\mathbf{I'}$, the main issue is to use the Poincar\'e inequality to fix the slope, and the proof is similar to \eqref{eq.TermI}
			\begin{equation}\label{eq.paraDecom2Term1}
				\begin{split}
					\vert \mathbf{I'} \vert &\leq C  3^n\Ll(\int_0^t \Er\Ll[\int_{\Rd} \vert \nabla (\tilde G_s - G_s)\vert^2 \, \d \mu \Rr] \, \d s\Rr)^{\frac{1}{2}} \Ll(\int_{0}^t \sum_{z \in \Z_{n}} \int_{z+\cu_n} \vert\nabla^2 g_s\vert^2 \, \d m  \, \d s\Rr)^{\frac{1}{2}}\\
					&=C 3^n\Ll(\int_0^t \Er\Ll[\int_{\Rd} \vert \nabla (\tilde G_s - G_s)\vert^2 \, \d \mu \Rr] \, \d s\Rr)^{\frac{1}{2}} \Ll(\int_{0}^t \norm{\nabla^2 g_s}^2_{L^2}  \, \d s\Rr)^{\frac{1}{2}}\\
					&\leq C 3^n\Ll(\int_0^t \Er\Ll[\int_{\Rd} \vert \nabla (\tilde G_s - G_s)\vert^2 \, \d \mu \Rr] \, \d s\Rr)^{\frac{1}{2}}  \norm{\nabla g}_{L^2}  \\
					&\leq \frac{1}{4}\int_0^t \Er\Ll[\int_{\Rd} \vert \nabla (\tilde G_s - G_s)\vert^2 \, \d \mu \Rr] \, \d s + C 3^{2n}\norm{\nabla g}^2_{L^2}. 
				\end{split}
			\end{equation}
			In the third line, we apply the classical heat kernel decay \eqref{eq.heatkernelClassic}. In the fourth line, the Young's inequality is used, and the first term there can be compensated by part of the {\lhs} in \eqref{eq.paraDecom}. 
			
			For the term $\mathbf{II'}$, we apply the flux replacement argument like \eqref{eq.TermII1} and \eqref{eq.TermII2}
			\begin{equation}\label{eq.paraDecom2Term2}
				\begin{split}
					\vert \mathbf{II'}\vert &\leq C \int_0^t 3^{-\alpha n} \norm{\nabla g_s}_{L^2} \Er^{\frac{1}{2}}\Ll[\int_{\Rd} \vert \nabla (\tilde G_s - G_s)\vert^2 \, \d \mu \Rr]  \, \d s \\
					&\leq \int_0^t C 3^{- 2\alpha n} \norm{\nabla g_s}^2_{L^2}  \, \d s + \frac{1}{4}\int_0^t \Er\Ll[\int_{\Rd} \vert \nabla (\tilde G_s - G_s)\vert^2 \, \d \mu \Rr] \, \d s\\
					&\leq C 3^{-2\alpha n} \norm{g}^2_{L^2}  + \frac{1}{4}\int_0^t \Er\Ll[\int_{\Rd} \vert \nabla (\tilde G_s - G_s)\vert^2 \, \d \mu \Rr] \, \d s. 
				\end{split}
			\end{equation}
			We apply Young's inequality in the second line, and the classical heat kernel decay \eqref{eq.heatkernelClassic} in the third line. We put \eqref{eq.paraDecom2Term1}, \eqref{eq.paraDecom2Term2} back to \eqref{eq.paraDecom2Decom} and get 
			\begin{multline}\label{eq.paraDecom2}
				\Ll\vert \int_{0}^t \Er\Ll[\int_{\Rd} \nabla (\tilde G_s - G_s) \cdot (\a \nabla \tilde G_s - \ab \nabla \bar G_s) \, \d \mu \Rr]\Rr\vert \\
				\leq \frac{1}{2}\int_0^t \Er\Ll[\int_{\Rd} \vert \nabla (\tilde G_s - G_s)\vert^2 \, \d \mu \Rr] \, \d s + C 3^{-2\alpha n} \norm{g}^2_{L^2}  +  C 3^{2n}\norm{\nabla g}^2_{L^2}.
			\end{multline}
			Finally, we put \eqref{eq.paraDecom1}, \eqref{eq.paraDecom3}, \eqref{eq.paraDecom2} back to \eqref{eq.paraDecom} and obtain the desired result.
		\end{proof}

		\begin{remark}
			The result above is already enough for us to prove \eqref{e.two.pointScale} with the scaling and a good test function $g \in H^2 \cap L^1$. Let us define 
			\begin{align*}
				g^N(x) := N^{-\frac{d}{2}} g\Ll(\frac{x}{N}\Rr), \qquad G^N := \int_\Rd g \, \d \mu - \rho \int_{\Rd} g \d m.
			\end{align*}
			This scaling gives us 
			\begin{align*}
				\norm{G^N}^2_{\cL^2} = \rho \norm{g^N}^2_{L^2} = \rho \norm{g}^2_{L^2}.
			\end{align*}
			Then we calculate the density field correlation 
			\begin{align*}
				&\Ll\vert\Er \EEbmu\Ll[Y^N_t(g) Y^N_0(g)\Rr] - \rho \int_{\Rd \times \Rd} g(x)\Psi_{t}(x-y)g(y)\, \d x  \d y \Rr\vert\\
				&=  \Ll\vert \Er\Ll[G^N (P_{N^2t}-\Pb_{N^2t})G^N\Rr] \Rr\vert\\
				&\leq \norm{G^N}_{\cL^2} \norm{(P_{N^2t}-\Pb_{N^2t})G^N}_{\cL^2}\\
				&\leq C\Ll(3^{-\alpha n} + (N^2t)^{-\frac{1}{8}}\Rr)\norm{g}_{L^2} + 3^n N^{-1} \norm{\nabla g}_{L^2}+ 3^n (N^2t)^{\frac{5}{8}}N^{-2} \norm{\nabla^2 g}_{L^2}.
			\end{align*}
			With the scaling, we gain an extra factor $N^{-1}$ for the gradient. It suffices to choose $1 \ll 3^n \ll N^{\frac{3}{4}}$, for example, $3^n \simeq N^{\frac{1}{2}}$, then the error terms vanish when $N \to \infty$ with an explicit rate. 
		\end{remark}
		
		\subsection{Regularization effect}
		In this part, we will show that the diffusive regularization effect of the parabolic semigroup can help weaken the assumption $g \in H^2$ in Proposition~\ref{prop.TwoScalePara}.  
		
		\begin{lemma}\label{lem.regularization}
			For any $\tau, t > 0$, we have the following estimates
			\begin{align}\label{eq.regularization}
				\norm{\Pb_\tau P_t - P_t}_{\cL^2 \to \cL^2}  \leq \sqrt{\frac{\Lambda\tau}{t}} \quad \text{ and } \quad \norm{\Pb_\tau \Pb_t - \Pb_t}_{\cL^2 \to \cL^2} \leq \sqrt{\frac{\Lambda\tau}{t}} .
			\end{align}
		\end{lemma}
		\begin{proof}
			Using the parabolic equation \eqref{eq.semigroupA} associated to $\Pb_t$, for any $\tilde F \in \cL^2$
			\begin{align*}
				\norm{\Pb_\tau \tilde F}^2_{\cL^2} - \norm{\tilde F}^2_{\cL^2} + \int_0^\tau \Er\Ll[\int_{\Rd}  \nabla \Pb_s \tilde F \cdot \ab \nabla \Pb_s \tilde F \, \d \mu \Rr] \, \d s = 0,
			\end{align*}
			which implies that 
			\begin{align*}
				\norm{\Pb_\tau \tilde F}^2_{\cL^2} &\geq \norm{\tilde F}^2_{\cL^2} - \tau \Er\Ll[\int_{\Rd}   \nabla  \tilde F \cdot \ab  \nabla  \tilde F \, \d \mu \Rr] \\
				&\geq  \norm{\tilde F}^2_{\cL^2} - \Lambda \tau \Er\Ll[\int_{\Rd}  \vert \nabla  \tilde F \vert^2 \, \d \mu \Rr].
			\end{align*}
			Here we use the decay of the Dirichlet energy of the parabolic semigroup in Proposition~\ref{prop.ElementSemigroup}. Then we insert $\tilde{F} = P_t F$, and obtain 
			\begin{align}\label{eq.regularization1}
				\norm{\Pb_\tau P_t F}^2_{\cL^2} &\geq \norm{P_t F}^2_{\cL^2} - \Lambda \tau \Er\Ll[\int_{\Rd}  \vert \nabla  P_t F  \vert^2 \, \d \mu \Rr]\\
				\nonumber &\geq \norm{P_t F}^2_{\cL^2} - \frac{\Lambda \tau}{t} \norm{F}^2_{\cL^2}.
			\end{align}
			In the last line, we also use the semigroup property associated to $P_t$ in \eqref{eq.L2Decay}. This estimate results in that of $\norm{\Pb_\tau P_t - P_t}_{\cL^2 \to \cL^2}$
			\begin{align*}
				\norm{(\Pb_\tau P_t - P_t) F}^2_{\cL^2} &= \norm{\Pb_\tau P_t F}^2_{\cL^2} + \norm{P_t F}^2_{\cL^2} - 2 \norm{\Pb_{\frac{\tau}{2}} P_t F}^2_{\cL^2}\\
				&\leq 2 \norm{P_t F}^2_{\cL^2} - 2\Ll(\norm{P_t F}^2_{\cL^2} - \frac{\Lambda \tau}{2t}\norm{F}^2_{\cL^2} \Rr)\\
				&\leq \frac{\Lambda \tau}{t}\norm{F}^2_{\cL^2}.
			\end{align*}
			Here in the second line, we apply $\norm{\Pb_\tau P_t F}^2_{\cL^2} \leq \norm{P_t F}^2_{\cL^2}$ and the estimate \eqref{eq.regularization1} to $\norm{\Pb_{\frac{\tau}{2}} P_t F}^2_{\cL^2}$. A similar argument also works for $\norm{\Pb_\tau \Pb_t - \Pb_t}_{\cL^2 \to \cL^2}$.
		\end{proof}
		
		\begin{proof}[Proof of Theorem~\ref{t.two.point}]
			For two functions $f, g$ as in the statement, we define 
			\begin{align*}
				F(\mu) := Y_0(f) = \int_{\Rd} f \, \d \mu - \rho \int_{\Rd} f \, \d m,\\
				G(\mu) := Y_0(g) = \int_{\Rd} g \, \d \mu - \rho \int_{\Rd} g \, \d m.
			\end{align*}
			As $\Pr$ is a stationary measure for the process $(\bmu_t)_{t \geq 0}$, it suffices to treat the case $s = 0$, and \eqref{e.two.point} is equivalent to estimating $\Er[G (P_t -\Pb_t) F]$. We apply the regularization at first. Recall that we denote by $\bracket{\cdot, \cdot}_{\cL^2}$ the inner product of $\cL^2$ space
			\begin{align*}
				&\Ll\vert \bracket{G, (P_t -\Pb_t) F}_{\cL^2} \Rr\vert \\
				&\leq \Ll\vert \bracket{G, \Pb_\tau (P_t -\Pb_t) F}_{\cL^2} \Rr\vert + \norm{G}_{\cL^2} \norm{\Pb_\tau (P_t -\Pb_t) F - (P_t -\Pb_t) F}_{\cL^2}\\
				&\leq \Ll\vert \bracket{G, \Pb_\tau (P_t -\Pb_t) F}_{\cL^2} \Rr\vert + 2\sqrt{\frac{\Lambda\tau}{t}} \norm{G}_{\cL^2}  \norm{F}_{\cL^2}.
			\end{align*}
			In the second line, we apply Lemma~\ref{lem.regularization} to the second term. When $\tau \ll t$, the error paid is very small and it suffices to consider the regularized version. Then we use the self-adjoint property of the operator 
			\begin{align*}
				\bracket{G, \Pb_\tau (P_t -\Pb_t) F}_{\cL^2} = \bracket{ (P_t -\Pb_t) (\Pb_\tau G),   F}_{\cL^2}.
			\end{align*}
			The term $\Pb_\tau G$ gains some more regularity, and it still has an explicit expression 
			\begin{align*}
				\Pb_\tau G (\mu) = \int_{\Rd} g_\tau \, \d \mu - \rho \int_{\Rd} g_\tau \, \d m,
			\end{align*}
			with $g_\tau = \Psi_\tau \star g \in H^2(\Rd) \cap L^1(\Rd)$. Thus, we apply Proposition~\ref{prop.TwoScalePara} to it
			\begin{align*}
				\Ll\vert \bracket{G, \Pb_\tau (P_t -\Pb_t) F}_{\cL^2} \Rr\vert &\leq \norm{ (P_t -\Pb_t) (\Pb_\tau G)}_{\cL^2} \norm{F}_{\cL^2} \\
				&\leq C\Ll(\Ll(3^{-\alpha n} + t^{-\frac{1}{8}}\Rr)\norm{g_\tau}_{L^2} + 3^n \norm{\nabla g_\tau}_{L^2} + 3^n t^{\frac{5}{8}}\norm{\nabla^2 g_\tau}_{L^2}\Rr) \norm{f}_{L^2}\\
				&\leq C\Ll(3^{-\alpha n} + t^{-\frac{1}{8}} +  3^n \tau^{-\frac{1}{2}} + 3^n t^{\frac{5}{8}} \tau^{-1}\Rr)\norm{g}_{L^2} \norm{f}_{L^2}.
			\end{align*}
			From the second line to the third line, we apply the decay of classical heat kernel \eqref{eq.heatkernelClassic}. Combining all the estimates above, we obtain that 
			\begin{multline}\label{eq.ErrorSource}
				\Ll\vert \bracket{G, (P_t -\Pb_t) F}_{\cL^2} \Rr\vert \\
				\leq  C\Ll(3^{-\alpha n} + t^{-\frac{1}{8}} +  3^n \tau^{-\frac{1}{2}} + 3^n t^{\frac{5}{8}} \tau^{-1} + (\tau/t)^{\frac{1}{2}}\Rr)\norm{g}_{L^2} \norm{f}_{L^2}.
			\end{multline}
			With a choice of mesoscopic scales $1 \ll 3^n \ll t^{\frac{5}{8}} \ll \tau \ll t$, for example $\tau = t^{\frac{3}{4}}, 3^n \simeq t^{\frac{1}{16}}$ we obtain the desired result \eqref{e.two.point} with a parameter $\beta = \frac{\min\{\alpha, 1\}}{16}$.
		\end{proof}
		\begin{remark}
			The errors in \eqref{eq.ErrorSource} have clear interpretation as the regularization, the homogenization and the price to fix the local slope to homogenization.
		\end{remark}
		\begin{proof}[Proof of Proposition~\ref{prop.HomoSemigroup}] Since both $P_t$ and $\Pb_t$ are conservative semigroups, we denote by $\hat{F}:= F - \rho \int_{\Rd} f \, \d m$ and have $(P_t - \Pb_t )F = (P_t - \Pb_t )\hat{F}$. Then using the self-adjoint property of the semigroup, we have
			\begin{align*}
				\norm{(P_t - \Pb_t )F}^2_{\cL^2} = \bracket{\hat{F}, P_{2t} \hat{F}}_{\cL^2} + \bracket{\hat{F}, \Pb_{2t} \hat{F}}_{\cL^2} - 2\bracket{\Pb_t \hat{F}, P_t \hat{F}}_{\cL^2}.
			\end{align*}
			Theorem~\ref{t.two.point} applies to the first and third terms, then we obtain the estimate of the semigroup. 
		\end{proof}
		
		\section{Convergence rate of Green--Kubo formulas}
		\label{sec:5}
		In this section, we start by giving more details on how to make sense of the integral in the formula \eqref{eq.quanti.gk}, and then prove Theorem~\ref{t.quanti.gk}.
		We follow the discussion in Section~\ref{subsec.Dirichletform} to construct the semigroup $P_t^{(m)}$. We recall that $\cL^2(\cu_m)=L^2(\mmd(\Rd), \mcl F_{\cu_m}, \Pr)$ is the space of square-integrable $\mcl F_{\cu_m}$-measurable functions, and define the Dirichlet form $\Ee^{\a}_{\cu_m}$ by
		\begin{align}\label{eq.DirichletCube}
			\Ll\{\begin{aligned}
				& \D(\Ee^{\a}_{\cu_m}) := \cH^1_0(\cu_m),\\
				& \displaystyle{\forall u,v \in \D(\Ee^{\a}_{\cu_m}), \quad \Ee^{\a}_{\cu_m}(u,v) := \Er\Ll[\int_{\cu_m} \frac{1}{2}\nabla u \cdot \a \nabla v \, \d \mu \Rr].}
			\end{aligned}\Rr.
		\end{align}
		The Dirichlet form $(\Ee^{\a}_{\cu_m}, \D(\Ee^{\a}_{\cu_m}))$ is also associated to a unique self-adjoint positive semi-definite operator $-A^{(m)}$
		\begin{align}\label{eq.defAm}
			\Ll\{\begin{array}{ll}
				\D(\Ee^{\a}_{\cu_m}) & = \D(\sqrt{-A^{(m)}}), \\
				\Ee^{\a}_{\cu_m}(u,v) & = \bracket{\sqrt{-A^{(m)}}u, \sqrt{-A^{(m)}}v}_{\cL^2}.
			\end{array}\Rr.
		\end{align}
		We define the semigroup $P_t^{(m)} := e^{t A^{(m)}}$. It satisfies similar properties as in Proposition~\ref{prop.ElementSemigroup}, and in particular, for every $u \in \cL^2(\cu_m)$ and $t > 0$, we have $P_t^{(m)} u \in \cH^1_0(\cu_m)$.
		
		By the definition \eqref{eq.defFpm}, the functional $F_{p,m}$ is an element in $\cH^{-1}(\cu_m)$. We write $\bracket{\cdot, \cdot}_{\cH^{-1}, \cH^1_0}$ as shorthand for $\bracket{\cdot, \cdot}_{\cH^{-1}(\cu_m), \cH^1_0(\cu_m)}$, since $m$ is kept fixed throughout this section. For each $t > 0$, we define $P_t^{(m)} (F_{p,m}) \in \cL^2(\cu_m)$ by duality as the unique function in $\cL^2(\cu_m)$ such that 
		\begin{align}\label{eq.defPFpm}
			\forall f \in \cL^2(\cu_m), \qquad \bracket{P_t^{(m)} (F_{p,m}), f}_{\cL^2} = \bracket{F_{p,m}, P_t^{(m)} f}_{\cH^{-1}, \cH^1_0}.
		\end{align}
		This definition of $P_t^{(m)}(F_{p,m})$ relies on the Riesz representation theorem and on the observation that for each $t > 0$, the mapping $f \mapsto  \bracket{F_{p,m}, P_t^{(m)} f}_{\cH^{-1}, \cH^1_0}$ is a continuous linear form on $\cL^2(\cu_m)$, since by \eqref{eq.L2Decay},
		\begin{equation}\label{eq.L2FpmBound}
			\begin{split}
				\Ll\vert \bracket{F_{p,m}, P_t^{(m)} f}_{\cH^{-1}, \cH^1_0} \Rr\vert &\leq \norm{F_{p,m}}_{\cH^{-1}} \norm{P_t^{(m)} f}_{\cH^1} \\
				& \leq  t^{-\frac{1}{2}}\norm{F_{p,m}}_{\cH^{-1}} \norm{f}_{\cL^2}.
			\end{split}
		\end{equation}

		The next proposition collects additional properties satisfied by $F_{p,m}$ and $P_t^{(m)} (F_{p,m})$.
		\begin{proposition}\label{prop.Fpm}
			The following properties are valid.
			\begin{enumerate}
				\item With $\phi_{p,\cu_m}$ defined in \eqref{eq.corrector}, we have for every $ v$ in $\cH^1_0(\cu_m)$ that
				\begin{align}\label{eq.FpmCorrector}
					\bracket{F_{p,m}, v}_{\cH^{-1}, \cH^1_0} = \Ee^{\a}_{\cu_m}(\phi_{p,\cu_m}, v),
				\end{align}
				and 
				\begin{align}\label{eq.FpmContinuity}
					\lim_{t\to 0}\bracket{P_t^{(m)} (F_{p,m}), v}_{\cL^2} = \bracket{F_{p,m}, v}_{\cH^{-1}, \cH^1_0}. 
				\end{align}
				\item For every $t,s >0$, we have $P_{t+s}^{(m)} (F_{p,m}) = P_{t}^{(m)}( P_{s}^{(m)} (F_{p,m}))$.
				\item There exists $C(d) < +\infty$ such that for every $t >0$, we have that $P_t^{(m)} (F_{p,m})$ belongs to $\cH^1_0(\cu_m)$, is centered, and satisfies
				\begin{align}
					\norm{P_t^{(m)} (F_{p,m})}_{\cL^2} &\leq 4 \exp(-3^{-2m}t/C)  t^{-\frac{1}{2}}\norm{F}_{\cH^{-1}}, \label{eq.FpmL2Decay}\\
					\norm{P_t^{(m)} (F_{p,m})}_{\acH^1} &\leq 4 \exp(-3^{-2m}t/C)  t^{-1} \norm{F}_{\cH^{-1}}. \label{eq.FpmH1Decay}
				\end{align}
			\end{enumerate}
		\end{proposition}
		\begin{proof} 
			(1). Equation \eqref{eq.FpmCorrector} is a consequence of \eqref{eq.corrector} and \eqref{eq:VariationalFormula72}, which we combine with the definition of \eqref{eq.defFpm} and \eqref{eq.L2FpmBound} to obtain 
			\begin{align*}
				\bracket{P_t^{(m)} (F_{p,m}), v}_{\cL^2} = \bracket{F_{p,m}, P_t^{(m)} v}_{\cH^{-1}, \cH^1_0} = \Ee^{\a}_{\cu_m}(\phi_{p,\cu_m}, P_t^{(m)} v).
			\end{align*}
			The continuity of the Dirichlet energy in (2) of Proposition~\ref{prop.ElementSemigroup} implies that
			\begin{multline*}
				\lim_{t \to 0} \bracket{P_t^{(m)} (F_{p,m}), v}_{\cL^2} = \lim_{t \to 0} \Ee^{\a}_{\cu_m}(\phi_{p,\cu_m}, P_t^{(m)} v) \\
				=  \Ee^{\a}_{\cu_m}(\phi_{p,\cu_m}, v) = \bracket{F_{p,m}, v}_{\cH^{-1}, \cH^1_0}.
			\end{multline*} 
			
			(2). We test $P_{t}^{(m)}( P_{s}^{(m)} (F_{p,m}))$ with a function $f \in \cL^2(\cu_m)$
			\begin{align*}
				\bracket{P_{t}^{(m)}( P_{s}^{(m)} (F_{p,m})), f}_{\cL^2} &= \bracket{ P_{s}^{(m)} (F_{p,m}), P_{t}^{(m)} f}_{\cL^2}\\
				&= \bracket{F_{p,m}, P_{s}^{(m)} (P_{t}^{(m)} f)}_{\cL^2}\\
				&= \bracket{F_{p,m}, P_{t+s}^{(m)} f}_{\cL^2}.
			\end{align*}
			In the first line, we use the reversibility of $P_{t}^{(m)}$ with respect to $\Pr$, and then we apply the definition \eqref{eq.defFpm} from the first line to the second line. The result in the third line follows the semigroup property of $P_{t+s}^{(m)}$. Since it satisfies the characterization of $P_{t+s}^{(m)} (F_{p,m})$, the two quantities coincide. 
			
			(3). Using the property in (2), we have $P_t^{(m)} (F_{p,m}) = P_{t/2}^{(m)} \Ll(P_{t/2}^{(m)} (F_{p,m})\Rr)$, so the first semigroup $P_{t/2}^{(m)}$ maps the function from $\cL^2(\cu_m)$ to $\cH^1(\cu_m)$. $P_t^{(m)} (F_{p,m})$ is centered because we test \eqref{eq.defFpm} with constant function $f \equiv 1$ and it will give $0$ on the {\rhs}. Once we know the function is centered, we can apply the spectral gap inequality \eqref{e.Poincare} that
			\begin{align*}
				\frac{\d}{\d t} \norm{P_t^{(m)} (F_{p,m})}_{\cL^2}^2 &= - 2 \Ee^{\a}_{\cu_m}(P_t^{(m)} (F_{p,m}),P_t^{(m)} (F_{p,m}))\\ &\leq -\frac{3^{-2m}}{C} \norm{P_t^{(m)} (F_{p,m})}_{\cL^2}^2.
			\end{align*}
			This results in that $\norm{P_t^{(m)} (F_{p,m})}_{\cL^2}^2 \leq \exp(-3^{-2m}t/(2C)) \norm{P_{t/2}^{(m)} (F_{p,m})}_{\cL^2}^2$. Then we apply the estimate \eqref{eq.L2FpmBound} for $\norm{P_{t/2}^{(m)} (F_{p,m})}_{\cL^2}^2$ and obtain \eqref{eq.FpmL2Decay}. One more factor $t^{-1/2}$ can be gained when applying \eqref{eq.L2Decay} to the $\acH^1(\cu_m)$ estimate, which yields \eqref{eq.FpmH1Decay}.
		\end{proof}
		
		Using the properties above, we can justify that the integral in  \eqref{eq.quanti.gk} is well-defined. We write $\ell_{p,m}:= \int_{\cu_m} p \cdot x \, \d \mu(x)$ as the affine function in $\cu_m$.
		\begin{lemma}\label{lem.CurrentCorrelation}
			For every $m \in \N$ and $\lambda \geq 0$, the integral of the dynamic current-current correlation  
			\begin{multline}\label{eq.defCurrentCorrelation}
				\int_0^{+\infty}  e^{-\lambda t}  \la F_{p,m}, P_t^{(m)}(F_{p,m}) \rah \, \d t \\
				:= \lim_{\epsilon \to 0} \lim_{T \to \infty} \int_\epsilon^{T}  e^{-\lambda t}  \la F_{p,m}, P_t^{(m)}(F_{p,m}) \rah \, \d t,
			\end{multline}
			is well-defined. It is equal to 
			\begin{align}\label{eq.ValueCurrentCorrelation}
				\Er\Ll[\int_{\cu_m} \Ll(\frac{1}{2} p \cdot \a p -  \frac{1}{2} p \cdot \a \nabla U_\lambda\Rr) \, \d \mu \Rr],
			\end{align}
			where $U_\lambda$ is the unique centered solution in $\ell_{p,m} + \cH^1_0(\cu_m)$ such that   
			\begin{align}\label{eq.defUlambda}
				\forall v \in \cH^1_0(\cu_m), \qquad \lambda \bracket{U_{\lambda}, v}_{\cL^2} + \Ee^{\a}_{\cu_m}(U_{\lambda}, v) = \lambda \bracket{\ell_{p,m}, v}_{\cL^2}.
			\end{align} 
		\end{lemma}
		\begin{proof}
			Thanks to the exponential decay \eqref{eq.FpmL2Decay} and \eqref{eq.FpmH1Decay}, the integral 
			\begin{align*}
				V_{\lambda, \epsilon} := \lim_{T \to \infty} \int_\epsilon^{T} e^{-\lambda t} P_t^{(m)}(F_{p,m}) \, \d t,
			\end{align*}
			converges in $\cH^1_0(\cu_m)$, and the limit satisfies
			\begin{align}\label{eq.defVlambda}
				\forall v \in \cH^1_0(\cu_m), \qquad \lambda \bracket{V_{\lambda, \epsilon}, v}_{\cL^2} + \Ee^{\a}_{\cu_m}(V_{\lambda, \epsilon}, v) = \bracket{e^{-\lambda \epsilon}P_\epsilon^{(m)}(F_{p,m}), v}_{\cL^2}.
			\end{align}
			To see this, we test $V_{\lambda, \epsilon}$ with the {\lhs} of this equation, and apply \eqref{eq.defAm} 
			\begin{align*}
				\lambda \bracket{V_{\lambda, \epsilon}, v}_{\cL^2} + \Ee^{\a}_{\cu_m}(V_{\lambda, \epsilon}, v) &= \bracket{\int_\epsilon^\infty  (\lambda - A^{(m)}) e^{-t (\lambda - A^{(m)})}(F_{p,m}) \, \d t, v}_{\cL^2} \\
				&= \bracket{e^{-\lambda \epsilon}P_\epsilon^{(m)}(F_{p,m}), v}_{\cL^2}.
			\end{align*}		
			We test the equation \eqref{eq.defVlambda} with $V_{\lambda, \epsilon}$ and apply Young's inequality (for the case $\lambda=0$, we also need \eqref{e.Poincare}), then obtain that 
			\begin{align}\label{eq.VUniformBound}
				\sup_{\epsilon > 0}\norm{V_{\lambda, \epsilon}}_{\acH^1(\cu_m)} \leq C \norm{F_{p,m}}_{\cH^{-1}(\cu_m)}.
			\end{align}
			Therefore, the sequence admits at least one weak limit in the space $\cH^1_0(\cu_m)$. Moreover, when we pass $\epsilon \to 0$ in \eqref{eq.defVlambda} and apply \eqref{eq.FpmContinuity}, the weak limit $V_{\lambda}$ is unique which is characterized by 
			\begin{align}\label{eq.defV0}
				\forall v \in \cH^1_0(\cu_m), \qquad \lambda \bracket{V_{\lambda}, v}_{\cL^2} + \Ee^{\a}_{\cu_m}(V_{\lambda}, v) = \bracket{F_{p,m}, v}_{\cH^{-1}, \cH^1_0}.
			\end{align}
			Now we test $V_{\lambda, \epsilon}$ with $F_{p,m}$
			\begin{equation}\label{eq.DynamicLimit}
				\begin{split}
					\lim_{\epsilon \to 0} \lim_{T \to \infty} \int_\epsilon^{T}  e^{-\lambda t}  \la F_{p,m}, P_t^{(m)}(F_{p,m}) \rah \, \d t 
					&=\lim_{\epsilon \to 0} \bracket{F_{p,m}, V_{\lambda, \epsilon}}_{\cH^{-1}, \cH^1_0}\\
					&=\lim_{\epsilon \to 0} \Ee^{\a}_{\cu_m}(\phi_{p, \cu_m}, V_{\lambda, \epsilon})\\
					&= \Ee^{\a}_{\cu_m}(\phi_{p, \cu_m}, V_{\lambda})\\
					&= \bracket{F_{p,m},  V_{\lambda}}_{\cH^{-1}, \cH^1_0}.
				\end{split}
			\end{equation}
			Here we apply \eqref{eq.FpmCorrector} in the second line, and then pass to the limit using $V_{\lambda, \epsilon} \stackrel{\epsilon \to 0}{\rightharpoonup} V_{\lambda}$ in $\cH^1_0(\cu_m)$ from the second line to the third line. This justifies that \eqref{eq.defCurrentCorrelation} is well-defined. Finally, we define $U_\lambda := \ell_{p,m} + V_\lambda$, then \eqref{eq.defV0} becomes \eqref{eq.defUlambda} and we obtain the limit of the dynamic current correlation \eqref{eq.ValueCurrentCorrelation} using $U_\lambda$.
		\end{proof}

		We are now ready to prove the quantitative Green--Kubo formula in Theorem~\ref{t.quanti.gk}.
		
		\begin{proof}[Proof of Theorem~\ref{t.quanti.gk}]
			As discussed in Lemma~\ref{lem.CurrentCorrelation}, we use the stationarity  \eqref{e.def.amu} and insert \eqref{eq.ValueCurrentCorrelation} in the {\lhs} of \eqref{eq.quanti.gk} and it suffices to estimate 
			\begin{align}\label{eq.GreenKuboReduced} 
				\Ll\vert \frac{1}{2} p \cdot \ab p - \frac{1}{\rho |\cu_m|} \E_\rho \Ll[\int_{\cu_m} \frac{1}{2} p \cdot \a \nabla U_\lambda \, \d \mu \Rr]\Rr\vert,
			\end{align}
			with $U_\lambda$ defined in \eqref{eq.defUlambda}. We notice that the homogenized solution $\Ub \in \ell_{p,m} + \cH^1_0(\cu_m)$ solving  
			\begin{align}\label{eq.GKElliptic2}
				(\lambda - \nabla \cdot \ab \nabla )\Ub =  \lambda \ell_{p,m},
			\end{align}
			is exactly $\ell_{p,m}$ using a similar argument in \eqref{eq.IPP}. Then the main idea is to replace $U_\lambda$ in \eqref{eq.GreenKuboReduced} by its two-scale expansion to get the estimate.
			
			\smallskip
			\textit{Step~1: two-scale expansion.} We analyze the structure of $U_\lambda$ using the two-scale expansion technique \eqref{eq.defW} from Section~\ref{sec:3}, setting
			\begin{align}\label{eq.GKdefW}
				W := \ell_{p,m} + \sum_{i=1}^d \sum_{z \in \Z_{m,n}} p_i \phi_{e_i, z+\cu_n} = \ell_{p,m} + \sum_{z \in \Z_{m,n}} \phi_{p, z+\cu_n}. 
			\end{align}
			Here $p_i$ is the $i$-th coordinate of the vector $p$, and we use the linearity of the corrector as a function of $p$. Without loss of generality, we suppose $\vert p \vert = 1$. We can establish the following equation using \eqref{eq.defUlambda} and \eqref{eq.GKElliptic2} 
			\begin{align*}
				(\lambda - \nabla \cdot \a \nabla )(W - U_\lambda) &= 	(\lambda - \nabla \cdot \a \nabla )W - 	(\lambda - \nabla \cdot \ab \nabla )\Ub\\
				&= \lambda (W-\Ub) - \nabla \cdot (\a \nabla W - \ab \nabla \Ub).
			\end{align*}
			We test the equation above with $(W-U_\lambda)$ as their affine parts cancel and ${(W-U_\lambda)} \in \cH^1_0(\cu_m)$. This gives us 
			\begin{multline}\label{eq.GKEllipticDecom}
				\lambda \norm{W-U_\lambda}^2_{\cL^2} + \Er \Ll[ \int_{\cu_m} \nabla (W - U_\lambda) \cdot \a \nabla (W - U_\lambda) \, \d \mu\Rr]   \\
				\leq \lambda \vert \Er[(W-U_\lambda)(W-\Ub)]\vert + \Ll\vert \Er\Ll[\int_{\cu_m} \nabla (W-U_\lambda) \cdot  (\a \nabla W - \ab \nabla \Ub) \, \d \mu\Rr]\Rr\vert. 
			\end{multline}
			Its first term can be estimated by following Step~3 in the proof of Proposition~\ref{prop.TwoScaleElliptic}, 
			\begin{align}\label{eq.GKL2}
				\nonumber \lambda \vert \Er[(W-U_\lambda)(W-\Ub)]\vert  &\leq \frac{\lambda}{2} \norm{W-U_\lambda}^2_{\cL^2} + \frac{\lambda}{2} \norm{W-\Ub}^2_{\cL^2} \\
				&\leq \frac{\lambda}{2} \norm{W-U_\lambda}^2_{\cL^2} + \frac{C \lambda 3^{2n}}{2} \norm{\ell_{p,m}}^2_{\acH^1(\cu_m)}. 
			\end{align}
			For the second term on the {\rhs} of \eqref{eq.GKEllipticDecom}, we follow Step~2 in the proof of Proposition~\ref{prop.TwoScaleElliptic}, especially the estimate \eqref{eq.TwoScaleMixed}, to get
			\begin{equation}\label{eq.GKH1}
				\begin{split}
					&\Ll\vert \Er\Ll[\int_{\cu_m} \nabla (W-U_\lambda) \cdot  (\a \nabla W - \ab \nabla \Ub) \, \d \mu\Rr]\Rr\vert \\
					&\leq C 3^{-\alpha n}\norm{W-U_\lambda}_{\acH^1(\cu_m)} \norm{\ell_{p,m}}_{\acH^1(\cu_m)} \\
					& \leq \frac{C}{2}\norm{W-U_\lambda}^2_{\acH^1(\cu_m)}  + \frac{C 3^{-2\alpha n}}{2} \norm{\ell_{p,m}}^2_{\acH^1(\cu_m)}.
				\end{split}
			\end{equation}
			Notice that we do not have the error corresponding to the term $ 3^{2n} \vert\nabla \nabla \ub\vert^2$ in \eqref{eq.TwoScaleMixed}, because the second derivative of the affine function $\ell_{p,m}$ vanishes.
			
			We put the estimate \eqref{eq.GKH1} and \eqref{eq.GKL2} back to \eqref{eq.GKEllipticDecom}, which results in
			\begin{multline}\label{eq.GKTwoScale}
				\lambda \norm{W-U_\lambda}^2_{\cL^2} + \Er \Ll[ \int_{\cu_m} \nabla (W - U_\lambda) \cdot \a \nabla (W - U_\lambda) \, \d \mu\Rr] \\
				\leq C(\lambda 3^{2n} + 3^{-2\alpha n})\norm{\ell_{p,m}}^2_{\acH^1(\cu_m)}.
			\end{multline}

			\smallskip
			\textit{Step~2: replacement in the Green--Kubo formula.} Now, we replace $U_\lambda$ in \eqref{eq.GreenKuboReduced} by $W$ defined in \eqref{eq.GKdefW},
			\begin{align*}
				&\Ll\vert p \cdot \ab p - \frac{1}{\rho |\cu_m|} \E_\rho \Ll[\int_{\cu_m} p \cdot \a \nabla U_{\lambda}\, \d \mu \Rr]\Rr\vert \\
				& \leq \Ll\vert p \cdot \ab p - \frac{1}{\rho |\cu_m|} \E_\rho \Ll[\int_{\cu_m} p \cdot \a \nabla W \, \d \mu \Rr]\Rr\vert + \Ll\vert \frac{1}{\rho |\cu_m|} \E_\rho \Ll[\int_{\cu_m} p \cdot \a \nabla (W - U_\lambda) \, \d \mu \Rr]\Rr\vert \\
				& = \Ll\vert p \cdot \ab p - \frac{1}{\rho |\cu_n|} \E_\rho \Ll[\int_{\cu_n} p \cdot \a (p + \nabla \phi_{p, \cu_n})  \, \d \mu \Rr]\Rr\vert + \Ll\vert \frac{1}{\rho |\cu_m|} \E_\rho \Ll[\int_{\cu_m} p \cdot \a \nabla (W - U_\lambda) \, \d \mu \Rr]\Rr\vert. 
			\end{align*}
			Compared to $U_\lambda$, its two-scale expansion has a better structure and the first term above becomes the homogenization of the diffusion matrix. We apply \eqref{eq:QuantConvergence} to the first term, and \eqref{eq.GKTwoScale} to the second term, then obtain
			\begin{align}\label{eq.GreenKuboBoundMixed}
				\Ll\vert p \cdot \ab p - \frac{1}{\rho |\cu_m|} \E_\rho \Ll[\int_{\cu_m} p \cdot \a \nabla U_{\lambda}\, \d \mu \Rr]\Rr\vert \leq C(\lambda^{\frac{1}{2}} 3^{n} + 3^{-\alpha n}).
			\end{align}
			
			\smallskip
			\textit{Step~3: choice of the mesoscopic scale.} 
			Finally, we need a reasonable choice of the mesoscopic scale $3^n$ to quantify the error term in \eqref{eq.GreenKuboBoundMixed}.
			\begin{itemize}
				\item Regime $\lambda \in [1, \infty)$. This case is trivial, as the regularization from the resolvent is too strong to see any homogenization, thus our bound is of constant order. To analyze it, instead of applying \eqref{eq.GreenKuboBoundMixed}, we turn back directly to the last line \eqref{eq.DynamicLimit}, and use the $\cH^1$ estimate of $V_\lambda$ from \eqref{eq.defV0}.
				\item Regime $\lambda \in (3^{-2(1+\alpha)m}, 1)$. In the regime $(3^{-2(1+\alpha)m}, 1)$, the effect of regularization is still strong, but it allows us to see the homogenization in a mesoscopic scale roughly around $\lambda^{-\frac{1}{2}}$. In \eqref{eq.GreenKuboBoundMixed}, we choose $3^{n} \simeq \lambda^{- \frac{1}{2(1+\alpha)}}$, and obtain the convergence rate $\lambda^{\frac{\alpha}{2(1+\alpha)}}$. 
				\item Regime $\lambda \in [0, 3^{-2(1+\alpha)m}]$. In this regime, as the regularization is much smaller than the spectral gap of the domain, the error term from the regularization $\lambda^{\frac{1}{2}} 3^{n}$ in \eqref{eq.GreenKuboBoundMixed} is negligible compared to that from the homogenization $3^{-\alpha m}$, and the latter dominates the convergence rate of the Green--Kubo formula.
			\end{itemize}
		\end{proof}

		\bigskip	
		\noindent \textbf{Acknowledgements}
		The research of CG is supported in part by the National Key R\&D Program of China (No. 2023YFA1010400) and NSFC (No. 12301166), and part of this project was developed during his visit at TSIMF and RIMS. We would like to thank Zhen-Qing Chen, Xiaodan Li and Xiangchan Zhu for helpful discussions.

		\bibliographystyle{plain}
		\bibliography{KawasakiRef}
		
	\end{document}